\documentclass{amsart}

\usepackage{amssymb,latexsym}
\usepackage{graphicx}
\usepackage{subfigure}
\usepackage{enumitem}
\usepackage{url}
\usepackage{caption}
\usepackage[all]{xy}
\usepackage{hyperref}
\usepackage{setspace}

\newtheorem{theorem}{Theorem}[section]
\newtheorem{lemma}[theorem]{Lemma}
\newtheorem{proposition}[theorem]{Proposition}
\newtheorem{corollary}[theorem]{Corollary}
\newtheorem{assertion}[theorem]{Assertion}
\newtheorem*{step1}{Step~$\mathbf{1}$}
\newtheorem*{stepk}{Step~$\boldsymbol{k}$}
\newtheorem{classification}[theorem]{Classification}
\newtheorem{classicalfact}[theorem]{Classical Fact}

\theoremstyle{remark}
\newtheorem*{diffpfradrays}{Proof of Lemma \ref{l:radrays} for \textsc{diff}}
\newtheorem*{plpfradrays}{Proof of Lemma \ref{l:radrays} for \textsc{pl}}
\newtheorem*{gasketuniquenessproof}{Proof of Lemma \ref{l:gasketuniqueness}}
\newtheorem*{pfoflemmacfg}{Proof of Lemma \ref{cfg}}
\newtheorem*{pfoflemmaTG}{Proof of Lemma \ref{l:TGisgasket}}
\newtheorem*{pfofourunknottingtheoremmeq2}{Proof of the HLT (Theorem~\ref{t:ourunknottingtheorem}) for $m=2$}
\newtheorem*{pfofourunknottingtheoremmeq3pldiff}{Proof of the HLT (Theorem~\ref{t:ourunknottingtheorem}) for $m\geq3$ and \textsc{cat}$=$\textsc{pl} or \textsc{diff}}
\newtheorem*{pfofourunknottingtheoremmeq3top}{Proof of the HLT (Theorem~\ref{t:ourunknottingtheorem}) for $m\geq3$ and \textsc{cat}$=$\textsc{top}}
\newtheorem*{pfofdocilethmmgt3pldiff}{Proof of the MHLT (Theorem~\ref{docilethm}) for $m>3$ and \textsc{cat}$=$\textsc{pl} or \textsc{diff}}
\newtheorem*{pfofdocilethmmgt3top}{Proof of the MHLT (Theorem~\ref{docilethm}) for $m>3$ and \textsc{cat}$=$\textsc{top}}
\newtheorem*{pfofdocilethmmeq2}{Proof of the MHLT (Theorem~\ref{docilethm}) for $m=2$}
\newtheorem*{case1comp}{Case $\left|S\right|=1$}
\newtheorem*{case2comp}{Case $\left|S\right|=2$}
\newtheorem*{casegt2comp}{Proof of the MHLT (Theorem~\ref{docilethm}) for \textsc{cat}$=$\textsc{top} and $\left|S\right|>2$}
\newtheorem*{slabproof}{Proof of the Slab Theorem}
\newtheorem*{proofarc}{Proof of the Arc Flattening Lemma}
\newtheorem*{onpfofass}{Proof of Assertion~\ref{goodballnhbd}}
\newtheorem*{pfofC}{Proof of Proposition~\ref{propC} from Propositions~\ref{propA} and~\ref{propB}}
\newtheorem*{pfofmrt}{Proof of the $2$-MRT (Theorem~\ref{mrt})}
\newtheorem*{pf2dschoen}{Proof of \textsc{diff} Schoenflies for $m=2$}
\newtheorem*{pfhlt}{Proof of \textsc{diff} HLT (Theorem~\ref{t:ourunknottingtheorem}) for $m=2$}
\newtheorem*{sgp}{Sketch of a geometric proof of Assertion~\ref{2ga}}
\newtheorem*{sctp}{Sketch of a classical topological proof of Assertion~\ref{2ga}}
\newtheorem*{poc}{Proof of Classification (in outline)}
\newtheorem*{pfhlta1}{Proof of Assertion~\ref{hlta1}}
\newtheorem*{pfhlta2}{Proof of Assertion~\ref{hlta2}}
\newtheorem*{pfhenrya1}{Proof of Assertion~\ref{henrya1}}
\newtheorem*{pfhenrya2}{Proof of Assertion~\ref{henrya2}}
\newtheorem*{noqedproof}{Proof}
\newtheorem*{pfgrt}{Proof of the GRT (Corollary~\ref{grt})}
\newtheorem*{pfofequivcond}{Proof of Proposition~\ref{equivcond}}

\theoremstyle{definition}
\newtheorem{remark}[theorem]{Remark}
\newtheorem{remarks}[theorem]{Remarks}
\newtheorem{observation}[theorem]{Observation}
\newtheorem{notation}[theorem]{Notation}

\def\Int{{\textnormal{Int}}}
\def\int{{\textnormal{int}}}
\def~{\kern.2em }
\DeclareCaptionLabelFormat{ss}{#1~#2}
\def\ss{\vrule height 0 pt width 0 pt depth 5 pt}
\def\cs{\mathbin{\sharp}}
\def\csinf{\mathbin{\natural}}
\def\CSI/{{\textnormal{CSI}}}
\def\Hull{{\textnormal{Hull}}}

\def\R{{\mathbb R}}
\def\Z{{\mathbb Z}}
\def\Bd{\partial}

\font\ssem=cmssi10
\newcommand{\dfb}[1]{{\ssem #1}}

\font\cuf=cmtt8
\newcommand{\curl}[1]{{\cuf #1}}

\newcommand{\fg}[1]{{\left\|  #1 \right\|}}
\newcommand{\nss}[1]{{\vskip 1.5ex \noindent {\bfseries \normalsize #1.}\hfill\vglue 0.75ex}}

\def \Item#1!{\noindent\hbox to 18pt{\bf\kern2pt #1 \hss\ignorespaces}}

\begin{document}
\title[\textnormal{CSI} and Hyperplane Unknotting]{Connected Sum at Infinity and Cantrell-Stallings hyperplane unknotting}
\author[J. Calcut]{Jack S. Calcut}
\address{Department of Mathematics\\
         Oberlin College\\
         Oberlin, OH 44074}
\email{jack.calcut@oberlin.edu}
\urladdr{\href{http://www.oberlin.edu/faculty/jcalcut/}{\curl{http://www.oberlin.edu/faculty/jcalcut/}}}
\author[H. King]{Henry C. King}
\address{Department of Mathematics, University of Maryland\hfill\break
         \indent College Park, MD 20742-4015}
\email{hck@math.umd.edu}
\urladdr{\href{http://www.math.umd.edu/~hck/}{\curl{http://www.math.umd.edu/\textasciitilde hck/}}}
\author[L. Siebenmann]{Laurent C. Siebenmann}
\address{Laboratoire de Math\'ematique, B\^at 425, Bur 234, \hfill\break \indent Universit\'e de Paris-Sud\\
         91405-Orsay, France}
\email{lcs7777@gmail.com}
\urladdr{\href{http://topo.math.u-psud.fr/~lcs/}{\curl{http://topo.math.u-psud.fr/\textasciitilde lcs/}}\hfill\break
					\vfill
					\hrule
					\hrule height 3pt width 0pt
					\noindent \textrm{This preprint, dated 12 October 2010, improves on several earlier versions dating back as far as 2004;
														it will be posted at} \href{http://lcs98.free.fr/biblio/}{\curl{http://lcs98.free.fr/biblio/}}
										\textrm{along with any significant update.}
					\eject}

\dedicatory{Dedicated to Ljudmila V. Keldysh and the members of her topology seminar\\on the occasion of the centenary of her birth.{\hspace{-1pt}${}^1$}}
    
\date{September 17, 2008 (revised October 12, 2010)}
\maketitle

\footnotetext[1]{See~\cite{chernavsky05}.}

\section{Introduction}\label{s:intro}

We give a general treatment of the somewhat unfamiliar operation on manifolds called \emph{connected sum at infinity} or \textnormal{CSI} for short. A driving ambition has been to make the geometry behind the well definition and basic properties of \textnormal{CSI} as clear and elementary as possible. \textnormal{CSI} then yields a very natural and elementary proof of a remarkable theorem of J.C.~Cantrell and J.R.~Stallings~\cite{cantrell,stallings65}. It asserts unknotting of \textsc{cat} embeddings of $\mathbb{R}^{m-1}$ in $\mathbb{R}^{m}$ with $m \ne 3$, for all three classical manifold categories: topological (=~\textsc{top}), piecewise linear (=~\textsc{pl}), and differentiable (=~\textsc{diff}) --- as defined for example in~\cite{kirbysiebenmannbook}. It is one of the few major theorems whose statement and proof can be the same for all three categories. We give it the acronym HLT, which is short for ``Hyperplane Linearization Theorem'' (see Theorem~\ref{t:ourunknottingtheorem} plus~\ref{maintheorem}).\ss

We pause to set out some common conventions that are explained in~\cite{kirbysiebenmannbook} and in many textbooks. By default, spaces will be assumed \emph{metrizable}, and \emph{separable} (i.e. having a countable basis of open sets). Simplicial complexes will be unordered. A \textsc{pl} \emph{space} (often called a \emph{polyhedron}) has a maximal family of \textsc{pl} compatible triangulations by locally finite simplicial complexes. \textsc{cat} \emph{submanifolds} will be assumed \emph{properly embedded} and \textsc{cat} \emph{locally flat}.\ss

This Cantrell-Stallings unknotting theorem (=~HLT) arose as an enhancement of the more famous Schoenflies theorem initiated by B.~Mazur \cite{mazurbams} and completed by M.~Brown~\cite{brown60,brown62}. The latter asserts \textsc{top} unknotting of \textsc{top} codimension~$1$ spheres in all dimensions: any locally flatly embedded $(m-1)$-sphere in the $m$-sphere is the common frontier of a pair of embedded $m$-balls whose union is $S^{m}$. This statement is cleaner inasmuch as dimension $3$ is not exceptional. On the other hand, its proof is less satisfactory, since it does not apply to the parallel \textsc{pl} and \textsc{diff} statements. Indeed, for \textsc{pl} and \textsc{diff}, one requires a vast medley of techniques to prove the parallel statement, leaving quite undecided the case $m=4$, even today.\ss

The proof of this \textsc{top} Schoenflies theorem immediately commanded the widest possible attention and opened the classical period of intense study of \textsc{top} manifolds. There is an extant radio broadcast interview of R.~Thom in which he states that, in receiving his Fields Medal in 1958 in Edinburgh for his cobordism theories~\cite{thom} 1954, he felt that they were already being outshone by J.W.~Milnor's exotic spheres~\cite{milnorexoticsphere} 1956 and the Schoenflies theorem breakthrough of Mazur just then occurring.\ss

At the level of proofs, the Cantrell-Stallings theorem is perhaps the more satisfactory. The \textsc{top} proof we present is equally self contained and applies (with some simplifications) to \textsc{pl} and \textsc{diff}. At the same time, Mazur's original infinite process algebra is the heart of the proof. Further, dimension $3$ is not really exceptional. Indeed, as Stallings observed, provided the theorem is suitably stated, it holds good in all dimensions.\footnote{Stallings deals with \textsc{diff} only; his proof~\cite{stallings65} differs significantly from ours, but one can adapt it to \textsc{pl} and probably to \textsc{top}.} Finally, its \textsc{top} version immediately implies the stated \textsc{top} Schoenflies theorem. We can thus claim that the Cantrell-Stallings theorem, as we present it, is an enhancement of the \textsc{top} Schoenflies theorem that has exceptional didactic value.\ss

In dimensions $>3$, it is tempting to believe that there is a well defined notion of \CSI/ for open oriented \textsc{cat} manifolds with just one end, one that is independent of auxiliary choices in our definition of \textnormal{CSI} -- notably that of a so-called \emph{flange} (see Section~\ref{s:csi}) in each summand, or equivalently that of a proper homotopy class of maps of $[0,\infty)$ to each summand. It has been known since the 1980s~\cite{geoghegan} that such a proper homotopy class is unique whenever the fundamental group system of connected neighborhoods of infinity is Mittag-Leffler (this means that the system is in a certain sense equivalent to a sequence of group surjections). More recently~\cite[pp.~369--371]{geoghegan}, it has been established that there are uncountably many such proper homotopy classes whenever the Mittag-Leffler condition fails; given one of them, all others are classified by the non-null elements of the (first) derived projective limit of the fundamental group system at infinity. This interesting classification does not readily imply that rechoice of flanges can alter the underlying manifold isomorphism type of a \CSI/ sum in the present context; however, in a future publication, we propose to show that it can indeed.\ss

A classification of \textsc{cat} \emph{multiple} codimension~$1$ hyperplane embeddings in $\mathbb{R}^{m}$, for $m\neq3$, will be established in Section~\ref{s:multiplehyperplanes} showing they are classified by countable simplicial trees with one edge for each hyperplane. This result is called the Multiple Hyperplane Linearization Theorem, or 
MHLT for short (see Theorem~\ref{docilethm}). For \textsc{top} and $m>3$, its proof requires the Slab Theorem of C.~Greathouse~\cite{greathouse2}, for which we include a proof, that (inevitably) appeals to the famous Annulus Theorem. For dimension
$m=2$, MHLT can be reduced to classical results of Schoenflies and K\'er\'ekjart\'o which imply a classification of all separable contractible surfaces with nonempty boundary. See end of Section~\ref{s:multiplehyperplanes} for an outline and the lecture notes~\cite{siebenmann08} for the details.
However, we explain in detail a more novel proof that uses elementary Morse-theoretic methods to directly classify \textsc{diff} multirays in $\mathbb{R}^2$ up to ambient isotopy (see Theorem~\ref{mrt} and Remark~\ref{radrays2d}). The same method can be used to make our $2$-dimensional results largely bootstrapping.\ss

The high dimensional MHLT (Theorem~\ref{docilethm}) is the hitherto unproved result that brought this article into being! Indeed, the first two authors queried the third concerning an asserted classification for $m>3$ in Theorem~10.10, p.~117 of~\cite{siebenmannthesis}, that is there both unproved and misstated. This simplicial classification is used in~\cite{calcutking} to make certain noncompact manifolds real algebraic.\ss

As is often the case with a general notion, particular cases of \textnormal{CSI}, sometimes called \emph{end sum}, have already appeared in the literature. Notably, R.E.~Gompf~\cite{gompf} used end sum for \textsc{diff} $4$-manifolds homeomorphic to $\mathbb{R}^4$ and R.~Myers~\cite{myers} used end sum for $3$-manifolds. The present paper hopefully provides the first general treatment of \textnormal{CSI}. However, we give at most fleeting mention of \CSI/ for dimension $2$, because, on the one hand, its development would be more technical (non-abelian, see Remark~\ref{radrays2d} and~\cite{stallings62groups}), and on the other, its accomplishments are meager.\ss

This paper is organized as follows. Section~\ref{s:csi} defines \CSI/ and states its basic properties. Section~\ref{s:regnhbds} is a short discussion of certain \textsc{cat} regular neighborhoods of noncompact submanifolds. Sections~\ref{s:raysgaskets} and~\ref{s:proofbasiccsiprops} prove the basic properties of \CSI/. Section~\ref{s:csicantrellstallings} uses \CSI/ to prove the Cantrell-Stallings hyperplane unknotting theorem (=~HLT, Theorem~\ref{t:ourunknottingtheorem}). Section~\ref{s:rayunknot} applies results of Homma and Gluck to \textsc{top} rays to derive Cantrell's HLT (=~Theorem~\ref{maintheorem} for \textsc{top}). Section~\ref{singmultrays} studies proper maps and proper embeddings of multiple copies of $[0,\infty)$. Section~\ref{s:multiplehyperplanes} classifies embeddings of multiple hyperplanes (=~MHLT, Theorem~\ref{docilethm}). It includes an exposition of C.~Greathouse's Slab Theorem,
and in conclusion some possibly novel proofs of the $2$-dimensional MHLT and related results classifying contractible $2$-manifolds with boundary.\ss

The reader interested in proofs of the
$2$-dimensional versions of the main theorems
HLT (Theorems~\ref{t:ourunknottingtheorem} and~\ref{maintheorem}) and MHLT (Theorem~\ref{docilethm}) will want to read the later
parts of Section~\ref{s:multiplehyperplanes}. There, three very
different proofs are discussed, all
independent of \CSI/.  The one that 
is also relevant to higher dimensions is
a Morse theoretic study of rays; for it, read
$2$-MRT (Theorem~\ref{mrt}).\ss

We authors believe the best way to assimilate the coming sections is to proceed as we did in writing them: namely, at an early stage, attempt to grasp in outline the proof in Section~\ref{s:csicantrellstallings} of the central theorem HLT (Theorem~\ref{t:ourunknottingtheorem}), and only then fill in the necessary foundational material. Later, pursue some of the interesting side-issues lodged in other sections.

\section{\textnormal{CSI}: Connected Sum at Infinity}\label{s:csi}

Connected sum at infinity \textnormal{CSI} will now be defined for suitably equipped, connected \textsc{cat} manifolds of the same dimension\footnote{Dimensions $\leq 2$ seem to lack enough room to make \textnormal{CSI} a fruitful notion.} $\geq 3$. The most common forms of connected sum are the usual connected sum \textnormal{CS} and connected sum along boundary \textnormal{CSB}; we assume some familiarity with these. All three are derived from disjoint sum by a suitable geometric procedure that produces a new connected \textsc{cat} manifold. \textnormal{CSI} is roughly what happens to manifold interiors under \textnormal{CSB}.\ss

Recall that, to ensure well definition, \textnormal{CS} and \textnormal{CSB} both require some choices and technology, particularly for \textsc{top}. \textnormal{CS} requires choice of an embedded disk and appeals to an ambient isotopy classification of them; for \textsc{top} this classification requires the (difficult) Stable Homeomorphism Theorem (=~SHT), which will be discussed in Section~\ref{s:multiplehyperplanes}. \textsc{CSB} requires distinguished and oriented boundary disks where the \textnormal{CSB} is to take place. Since any \textnormal{CSB} operation induces a \textnormal{CS} operation of boundaries, it is clear that the extra boundary data for \textnormal{CSB} is essential for its well definition -- as dimension 3 already shows.\footnote{For example, let $X={S^{1}}\times{D^{2}}$ and $Y=X-\text{Int}D^{3}$ where $D^{3}$ is a small round disk in $\Int X$. The CSB operation on X and Y can produce two manifolds with non-homeomorphic boundaries.} The definition of \textnormal{CSI} has similar problems, and this imposes the notion of a \emph{flange}, which we define next.\ss

In any \textsc{cat}, connected, noncompact $m$-manifold $M$, one can choose a \textsc{cat}, codimension~$0$, proper, \emph{oriented} submanifold $P\subset\Int{M}$ that is \textsc{cat} isomorphic to the closed upper half space $\mathbb{R}^{m}_{+}$. For example, $P$ can be derived from a suitably defined \emph{regular neighborhood} of a \dfb{ray} $r$, where a ray is by definition a (proper) \textsc{cat} embedding of $[0,\infty)$. Such a $P$ with its orientation is called a \dfb{CSI flange}, or (for brevity) a \dfb{flange}. The pair $(M,P)$ is called a \dfb{CSI pair} or synonymously a \dfb{flanged manifold}. Often a single alphabetical symbol like $N$ will stand for a flanged manifold; then $\left|N\right|$ will denote the underlying manifold (flange forgotten). Thus, when $N = (M,P)$, one has $\left|N\right| := M$.\ss

In practice, rays and flanges are usually obvious or somehow given by the context, even in dimension $3$ where rays can be knotted. For example:\ss

\Item(i)! If $M$ is oriented (or even merely oriented near infinity) it is to be understood that the CSI flange orientation agrees with that of $M$ --- unless this requirement is explicitly waived.

\Item(ii)! If $M$ is a compact manifold with a connected boundary, then $\Int \, M$  has a preferred ray up to ambient isotopy; it arises as a fiber of a collaring of $\Bd M$ in $M$; this is because of a well known collaring uniqueness up to (ambient) isotopy that is valid in all three categories, cf.~\cite{kirbysiebenmannbook}.

\Item(iii)! With the data of (ii), suppose $\partial{M}$ is oriented. Then the preferred class of rays from (ii) and the isotopy uniqueness of regular neighborhoods (see Section~\ref{s:regnhbds}) provide a preferred (oriented) flange for $\Int M$ that is well defined up to ambient isotopy of $\Int M$. On the other hand, if $\partial{M}$ is non-orientable, then an ambient isotopy of $M$ can reverse the orientation of a regular neighborhood in $M$ of any point of $\partial{M}$; hence in this case also there is an (oriented) flange for $\Int M$ that is well defined up to ambient isotopy of $M$.

\Item(iv)! If $N$ has dimension $\leq 3$ and is isomorphic to the interior of a compact manifold with connected boundary, then once again $N$ has a preferred ray up to isotopy; this is because $N$ is irreducible near $\infty$ and irreducible h-cobordisms of dimension $\leq 3$ are products with $[0,1]$ (see~\cite{hempel}).\smallskip

A second ingredient for a \textnormal{CSI} sum of $m$-manifolds will be
a so-called \emph{gasket}.  The prototypical gasket is a
\dfb{linear gasket}; this is by definition a closed
subset of a certain model $\mathbb{H}^m$  of hyperbolic $m$-space
whose frontier is a nonempty collection of at most countably many disjoint codimension $1$ hyperplanes (see Figure~\ref{gaskets}).
\begin{figure}[h!]
\centering
\subfigure[\emph{A $2$-dimensional gasket.}]
{
    \label{gasket2D}
    \includegraphics[scale=0.8]{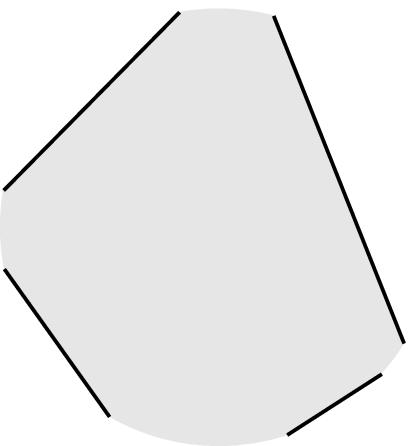}
}
\hspace{1cm}
\subfigure[\emph{A $3$-dimensional gasket.}]
{
    \label{gasket3D}
    \includegraphics[scale=0.75]{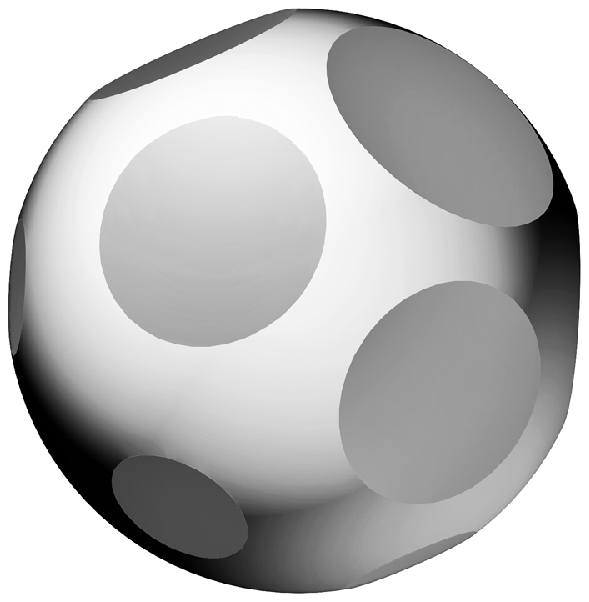}
}
\captionsetup{labelformat=ss,textfont=it}
\caption{Linear gaskets.}
\label{gaskets}
\end{figure}
We adopt Felix Klein's projective model of hyperbolic space; in
it, $\mathbb{H}^m$ is the open unit ball in $\mathbb{R}^m$, and each
codimension 1 hyperbolic hyperplane is by definition a nonempty
intersection with $\mathbb{H}^m$ of an affine linear $(m-1)$-plane in
$\mathbb{R}^m$. A \dfb{gasket} is by definition any oriented \textsc{cat}
$m$-manifold that is degree +1 \textsc{cat} isomorphic to a linear
gasket.\ss

\begin{remark}
A linear gasket is clearly simultaneously an 
oriented manifold of all three categories. The 
hyperbolic structure of $\mathbb{H}^m$ will occasionally be helpful.
However it can be treacherous for \textsc{pl}, since its
isometries are not all \textsc{pl}; they are projective
linear but mostly not affine linear (not even piecewise). Thus our mainstay will be
the \textsc{cat} structures inherited from $\mathbb{R}^m$.
\end{remark}

Consider an indexed set $\mu_i =(M_i, P_i)$ of \textnormal{CSI} pairs of dimension $m$,
where $i$ ranges over a nonempty finite or countable index set $S$.
The \textnormal{CSI} operation yields a \textnormal{CSI} pair $\omega =(W,Q)$
by the following construction (see Figure~\ref{csiop}).
\begin{figure}[h!]
	\centerline{\includegraphics[scale=0.8]{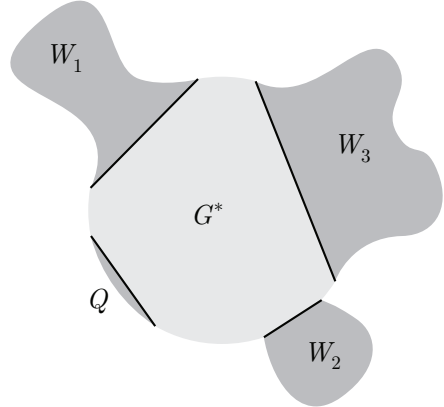}}
	\captionsetup{labelformat=ss,textfont=it}
	\caption{\textnormal{CSI} operation.}
	\label{csiop}
\end{figure}

Let $G^{\ast}$ be a linear gasket of the same dimension $m$,
with $\left|S\right|+1$ boundary components.
Each closed component of the complement of $G^{\ast}$ in $\mathbb{H}^m$
is a \textsc{cat} flange. We choose one, 
say $Q$, and write $G$ for the gasket $G^{\ast} \cup Q$. The flange
$Q$ will become the flange of $\omega$.\ss

A pair that is \textsc{cat} isomorphic to $(G,Q) := (G^{\ast} \cup Q,Q)$ as above will be called a \dfb{flanged gasket}. Equivalently, any CSI pair $(G',Q')$ where $G'$ and $G' - \Int{Q'}$ are both \textsc{cat} gaskets is by definition a flanged gasket.\ss

$W$ will now be formed by introducing identifications
in the disjoint sum:
\begin{equation}\label{disjointsum}
	\bigsqcup \hspace{-0.5mm} \left\{M_i \mid i\in S \right\} \ \sqcup \ G. \tag{\dag}
\end{equation}

We index by $S$ the $\left|S\right|$ components of $\partial{G}$, denoting
them by $H_i$, $i \in S$, and choose, for each, a \textsc{cat} degree
+1 embedding $\theta_i : P_i \to G^{\ast}$ onto an open collar
neighborhood of $H_i$ in $G^{\ast}$. Now form $W$ from the disjoint
sum~\eqref{disjointsum} by identifying $P_i$ to its image in $G^{\ast}$ under
$\theta_i$. Finally, $\omega := (W,Q)$ is by definition a \dfb{CSI sum} of the CSI pairs $\mu_i$, $i\in S$. \ss

We will call $G$ and $G^{\ast}$ respectively the 
\dfb{coarse gasket} and the \dfb{fine gasket} of the 
\textnormal{CSI} sum $\omega = (W,Q)$.

\begin{remark}
As a topological space, $W$ is somewhat more simply
expressed as the quotient space of the disjoint sum
\[
	\bigsqcup \hspace{-0.5mm} \left\{M_i - \Int{P_i} \mid i\in S \right\} \ \sqcup \ G
\]
by the identifications
\[
	\theta_i |_{\partial{P_i}} : \partial{P_i} \to H_i.
\]
In the \textsc{pl} category, these identifications induce a unique \textsc{pl}
manifold structure on $W$. But in the \textsc{diff}
category, the full collarings $\theta_i$ serve to provide a well defined differentiable manifold structure on $W$.
\end{remark}

\begin{theorem}\label{p:csiproperties}
The \textnormal{CSI} of a nonempty but countable (or finite) set of \textnormal{CSI} pairs of dimension $m\geq3$ enjoys the following properties: \smallskip

\Item(1)! From such a set $(M_{i},P_{i})$, $i \in S$, the \textnormal{CSI} construction above yields a \textnormal{CSI} pair $(W,Q)$ that is well defined up to \textsc{cat} isomorphism. Given a second such construction whose entries are distinguished by primes, a bijection $\varphi :S \to S'$, and, for each $i \in S$, an isomorphism of \textsc{cat} \textnormal{CSI} pairs $\psi_{i}:(M_{i},P_{i}) \to (M'_{\varphi(i)},P'_{\varphi(i)})$, there exists a \textsc{cat} isomorphism $\psi:(W,G,Q) \to (W',G',Q')$ that extends $\psi_{i}$ restricted to $M_{i}-\Int P_{i}$ for all $i \in S$. Furthermore, this $\psi$ is degree $+1$ as a map $G\to G'$, and induces an isomorphism of \textnormal{CSI} pairs $(W,Q) \to (W',Q')$. Thus, in addition to being well defined, the \textnormal{CSI} operation is commutative.

\Item(2)! The composite \textnormal{CSI} operation is associative.

\Item(3)! The \textnormal{CSI} operation has an identity element $\varepsilon = (\mathbb{R}^{m},\mathbb{R}^{m}_{+})$, and the infinite \textnormal{CSI} product $\varepsilon \varepsilon \varepsilon \cdots$ of copies of $\varepsilon$ is isomorphic to $\varepsilon$.
\end{theorem}

Precise definitions of composite \textnormal{CSI} operations and of their associativity
are given below in Section~\ref{s:proofbasiccsiprops}.\ss

\begin{notation}
Theorem~\ref{p:csiproperties} justifies the following notations for \CSI/ sums. If $\mathcal{M}$ is a nonempty but countable collection of flanged manifolds,
then $\CSI/(\mathcal M)$ can denote the flanged manifold resulting from the \CSI/ operation applied to these manifolds. And, in case $\mathcal M$ is an ordered sequence  $M_1$, $M_2$, $\ldots$, then 
$\CSI/(M_1,M_2,\ldots)$ and $\CSI/(\mathcal M)$ should be synonymous. An alternative to $\CSI/(M_1,M_2,...)$ introduced by Gompf~\cite{gompf} is
$M_1 \csinf M_2 \csinf \cdots$.
\end{notation}

\begin{remark}
In Theorem~\ref{p:csiproperties}, it is already striking that every infinite \textnormal{CSI} product yields a well defined \textnormal{CSI} pair (up to isomorphism). Nothing so strong is true for \textnormal{CS} or \textnormal{CSB} unless artificial limitations are imposed on the infinite connected sum operation. For example, in dimensions $m \geq 2$, an infinite \textnormal{CS} of any closed, connected, oriented $m$-manifold with itself could reasonably be defined so as to have any conceivable end space -- to wit any nonempty compact subset of the Cantor set.
\end{remark}

\begin{remark}\label{csiandcs} For \textsc{cat} = \textsc{diff} and \textsc{pl}, as observed in remarks at the beginning of this section, the interior of a \textsc{cat} compact $m$-manifold with nonempty connected boundary, has a privileged choice of flange
(up to ambient isotopy and orientation reversal). This lets us perceive some
near overlap of \textnormal{CSI} with the 
ordinary connected sum \textnormal{CS} as follows.  
Let us suppose $M$ is the connected sum $M_1 \cs M_2 \cs \cdots\cs M_k$ of a
\emph{finite} collection  $M_1,\ldots,M_k$ of
oriented connected closed $m$-manifolds, then \hbox{$M - (\textnormal{point})$} is 
\textsc{cat} isomorphic, preserving
orientation, to the flanged and oriented manifold $M'_1 \csinf M'_2 \csinf \cdots \csinf M'_k$ where  $M'_i$ 
is the manifold \hbox{$M_i - (\textnormal {point})$}
with a flange chosen whose orientation agrees
with that of $M_i$.   The reader is left to further
explore such relations between 
\textnormal{CSI} and \textnormal{CS}.
\end{remark}

\begin{remark}\label{flangeorientation} The last
remark above leads us to simple examples where 
reversal of a flange orientation changes the
underlying proper homotopy type of the CSI of two
flanged manifolds.  

It is a familiar fact that,
if $M$ is the complex projective plane (of real
dimension $4$), the ordinary connected sum $M
\cs (-M)$  has a signature zero cup product
bilinear form on the cohomology group
\[
\mathrm{H}^2(M \cs (-M)\,;\,\Z)     
              = \Z^2,\]
whilst  $M \cs M$ has form of signature $+2$ (the
sign $+$ becoming $-$ if we replace $M$ by $-M$).
It follows that $M \cs M$ and $M \cs -M$ 
are not homotopy equivalent.

Let $N$ be $M -\hbox{(point)}$, the complement of a
point in $M$, and forget the orientation of $N$,
but then consider two flanges $P_+$  and $P_-$ for
$N$ whose orientations agree with those of  $M$
and $-M$ respectively.  By Remark~\ref{csiandcs},
the CSI of $(N,P_+)$ and $(N,P_-)$ is  
 $(M \cs -M) - \hbox{(point)})$ 
 whose Alexandroff one-point compactification is 
 $(M \cs - M)$.  On the  other hand, the CSI of
$(N,P_+)$ and $(N,P_+)$  is $(M \cs M) - \hbox{(point)})$
whose one-point compactification is $(M \cs M)$. There
cannot be a \emph{proper} homotopy equivalence between
\[
 (M \cs -M) - \hbox{(point)}\quad \hbox{and}\quad 
       (M \cs M) - \hbox{(point)}
\]
because its one-point compactification
would  clearly be a homotopy equivalence
between $M \cs -M$ and $M \cs M$, which does
not exist.
\end{remark}

The proof of Theorem~\ref{p:csiproperties} will be mostly elementary. There is one important exception: the \textsc{top} version as presently stated requires the difficult Stable Homeomorphism Theorem (=~SHT) of~\cite{kirby69,freedmanquinn,edwards} to show that any homeomorphism of $\mathbb{R}^{m-1}$ is isotopic to a linear map. In contrast, for \textsc{cat}=\textsc{pl} or \textsc{cat}=\textsc{diff}, it is elementary that every \textsc{cat} automorphism of euclidean space is \textsc{cat} isotopic to a linear map (for \textsc{pl} see~\cite{rourkesanderson}, and for \textsc{diff} see~\cite[p.~34]{milnor97}).\ss

Happily, this dependence on a difficult result can and will be removed. Our tactic is to refine the definition of \textnormal{CSI} for \textsc{top} \textbf{requiring henceforth} (unless the contrary is indicated) that:
\smallskip

\noindent$\bullet$ The \textnormal{CSI} flange $P$ in each \textnormal{CSI} pair $(M,P)$ shall carry a preferred \textsc{diff} structure making $P$ \textsc{diff} isomorphic to $\mathbb{R}^{m}_{+}$, and, with respect to such structures, every \textnormal{CSI} pair isomorphism shall be \textsc{diff} on the flanges.
\smallskip

\noindent$\bullet$ Every gasket shall be equipped with a \textsc{diff} structure making it \textsc{diff} isomorphic to a linear gasket, and all of the identifications made in \textnormal{CSI} constructions shall be \textsc{diff} identifications with respect to these preferred \textsc{diff} structures.\ss

The magical effect of this refined definition is that the proof for \textsc{diff} of the basic properties of \textnormal{CSI} applies without essential changes to the \textsc{top} category. This is rather obvious if one thinks of \textsc{top} \textnormal{CSI} as being \textsc{diff} where all of the relevant action takes place. Consequently, for many cases of Theorem~\ref{p:csiproperties} we give little or no proof for the \textsc{top} category -- leaving the reader to do his own soul searching. Note that the above refinement could equally use \textsc{pl} in place of \textsc{diff}.

\section{Regular Neighborhoods}\label{s:regnhbds}

Regular neighborhoods will play a central technical role throughout this article. A short discussion of such \textsc{cat} neighborhoods, just sufficient for our uses, is given below.

\nss{PL Regular Neighborhoods}

\textsc{pl} regular neighborhood theory is a major feature of \textsc{pl} topology that is entirely elementary but not always simple. Such a theory was first formulated by J.H.C.~Whitehead \cite{whitehead39}, and then simplified and improved by E.C.~Zeeman~\cite{zeeman,hudsonzeeman} (see also~\cite{rourkesanderson}). We need the version of this theory that applies to possibly noncompact \textsc{pl} spaces; it is developed in~\cite{scott}. We now review some key facts.\ss

Let $X$ be a closed \textsc{pl} subspace of the \textsc{pl} space $M$. Neither is assumed to be compact, connected, nor even a \textsc{pl} manifold. Recall that $X$ is a subcomplex of some \textsc{pl} triangulation of $M$ by a locally finite simplicial complex. A regular neighborhood $N$ of $X$ can be defined to be a closed $\varepsilon$-neighborhood ($\varepsilon<0.5$) of $X$ in $M$ for the barycentric metric of some such triangulation of $M$.  The frontier of $N$ in $M$ is thus \textsc{pl} bicollared in $M$.\ss
	
We quickly recite some familiar facts. Any two regular neighborhoods $N$ and $N'$ of $X$ in $M$ are ambient isotopic fixing $X$. If $N_0$ is a regular neighborhood that lies in the (topological) interior $\int \, N$ of $N$ in $M$, then the triad
$(N - \int \, N_0; \delta N_0, \delta N)$
is \textsc{pl} isomorphic to the product triad $\delta N\times([0,1]; 0, 1)$ where $\delta$ indicates frontier in $M$.
Thus, if $N_0$ is contained in $\int \, N \cap \int \, N'$, and $U$ is a neighborhood of 
$N \cup N'$ in $M$, then the ambient isotopy carrying $N$ to $N'$ can be the identity on $N_0$ and on the complement of $U$.\ss

We will also use (in some special cases) two less familiar facts, namely Propositions~\ref{propA} and~\ref{propB}.

\begin{proposition}\label{propA}
If $N_i$ is a regular neighborhood
of $X_i$ in $M_i$ for $i=1$ and $i=2$, then 
 $N_1 \times N_2$ is a regular neighborhood of 
 $X_1 \times X_2$ in $M_1 \times M_2$. \qed
\end{proposition}

\begin{proposition}\label{propB}
Let $N$ be a properly embedded
$m$-submanifold of a \textsc{pl} $m$-manifold $M$ such that
$N \subset \Int\,M$, and let $X$ be a properly
embedded \textsc{pl} subspace of $M$ with $X \subset N$.
Then a sufficient condition for $N$ to be a regular
neighborhood of $X$ in $M$ is that $(N,X)$ be \textsc{pl}
isomorphic to a pair $(N',X')$ where $N'$ is a
regular neighborhood of $X'$ in a \textsc{pl} manifold
$M'$. \qed
\end{proposition}

\begin{proposition}\label{propC}
If $\rho : [0,\infty) \to \mathbb{R}^{m}_{+}$ is a proper linear ray embedding with
image $r$ in  $\Int\,\mathbb{R}^m_+$, then $\mathbb{R}_+^m$ is \textsc{pl} 
isomorphic fixing $r$ to a regular neighborhood of
$r$ in $\mathbb{R}^m_+$.
\end{proposition}

\begin{pfofC}
Adjusting $r$ by an affine linear
automorphism of $\mathbb{R}_+^m$, we may assume, without
loss of generality, that $r = 0 \times [2,\infty)$, where
the $0$ here denotes the origin of $\mathbb{R}^{m-1}= \partial{\mathbb{R}_+^m}$.\ss

For any real $\lambda > 0$ and integer $k >0$, let 
$B_\lambda^k :=[-\lambda,\lambda]^k$
and let $B_{<\lambda}^k := (-\lambda,\lambda)^k$.
Since each $B_{\lambda}^{m-1}$ is a regular neighborhood of
the origin, there exists a \textsc{pl} isomorphism for
any $\varepsilon \in (0,1) $~:
\[
\varphi : \mathbb{R}^{m-1} - B_{<\varepsilon}^{m-1} \to 
    [\varepsilon,\infty) \times  \partial{B_\varepsilon^{m-1}}
\]
extending the canonical identification
$\partial{B_\varepsilon^{m-1}} \cong \varepsilon \times \partial{B_\varepsilon^{m-1}}$.
Although $\varphi$ itself is not canonical, we regard it as an
identification.\ss

By Proposition~\ref{propA}, the product 
$B_1^{m-1}\times [1,\infty)$
is a regular neighborhood of $r$ in  $\mathbb{R}_+^m$. 
Thus, by Proposition~\ref{propB}, it certainly will suffice
to show that there exists, for some $\varepsilon\in(0,1)$, a \textsc{pl} isomorphism
\[
  h : \mathbb{R}^{m-1} \times [1,\infty) \to
  B_1^{m-1} \times [1,\infty) \tag{\dag}
\]
fixing $B_\varepsilon^{m-1} \times [2,\infty)$.\ss

For $\varepsilon\in (0,1)$, it is an elementary fact
about \textsc{pl} $2$-manifolds that there is a \textsc{pl}
isomorphism:
\[
 \theta: [\varepsilon,\infty) \times [1,\infty) 
   \to [\varepsilon,1] \times [1,\infty)
\]
fixing $\varepsilon \times [1,\infty)$.
Producting with the identity map of $\partial{B_\varepsilon^{m-1}}$,
and then extending by the identity over
$B_\varepsilon^{m-1} \times [1,\infty)$, we get the required
\textsc{pl} isomorphism $h$ for (\dag).  \qed
\end{pfofC}
	
\nss{DIFF Regular Neighborhoods}
	
There is a quite general elementary theory of smooth regular neighborhoods in \textsc{diff} manifolds. Unfortunately, it involves \textsc{pl}, is fastidious to develop, and currently occupies half of the monograph~\cite{hirschmazur} (see also~\cite{cairns67}). We therefore cobble together an ad hoc, but bootstrapping, notion of \textsc{diff} regular neighborhood for a \textsc{diff} ray $r$ in a \textsc{diff} $m$-manifold $M$ ($r$ is a proper \textsc{diff} embedded copy of $[0,\infty)$ in $M$). This notion will be derived from the well known notion of a tube about a submanifold and can be extended to most sorts of \textsc{diff} submanifolds.\ss
	
Let $p:V(r) \to r$ be the projection of a \textsc{diff} tube about the ray $r$. $V(r)$ is a \textsc{diff} submanifold of $M$ lying in $\text{Int}M$. It is a trivial \textsc{diff} bundle with projection $p$, fiber the unit $(m-1)$-disk, orthogonal group, and zero section the inclusion of $r$. It is not, however, a neighborhood of ${\partial}r = b$, nor a neighborhood of r itself. Also $V(r)$ has undesirable corners. To obtain an acceptable regular neighborhood of $r$, we trim $V(r)$ and add a cap along the butt end $p^{-1}(b)$ as follows (see Figure~\ref{diffrn}).
\begin{figure}[h!]
	\centerline{\includegraphics{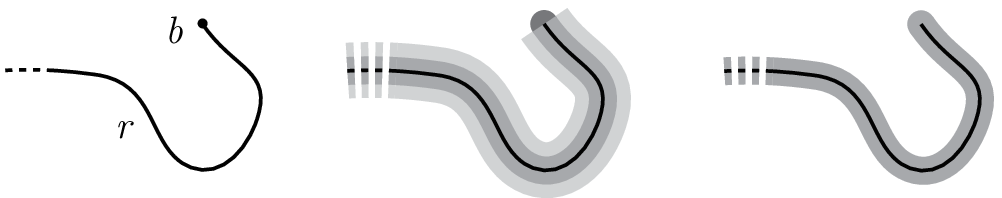}}
	\captionsetup{labelformat=ss,textfont={it,small},labelsep=newline,justification=centerlast,width=.88\textwidth}
	\caption{A \textsc{diff} regular neighborhood of a ray. The light gray (in center figure) indicates a tube about the ray $r$, and the dark gray (center and right) indicates the regular neighborhood of $r$.}
	\label{diffrn}
\end{figure}
Let $V'(r) \subset V(r)$ be the subbundle of disks of radius $1/2$. In the \textsc{diff} manifold with boundary (and corners) $M_{0} = M-\text{Int}V(r)$, the point $b = {\partial}r$ is a boundary point and the disk fiber $E^{m-1}$ of $V'$ at $b$ is a tube about $b$ in ${\partial}M_{0}$. There exists a tubular neighborhood $U(b)$ of $b$ in $M_{0}$. By \textsc{diff} tube uniqueness, we can arrange that $U(b) \cap {\partial}M_{0}$ coincides with $E^{m-1}$. Further, applying \textsc{diff} collaring existence and uniqueness to ${\partial}M_{0}$ in $M$, we can arrange that $T(r) = V'(r) \cup U(b)$ is smooth along ${\partial}E^{m-1}$, and hence is a \textsc{diff} submanifold of $M$ without corners. This $T(r)$ is, by definition, a \textsc{diff} \dfb{regular neighborhood} of $r$ in $M$.\ss
	
For a (proper) \textsc{diff} submanifold $L$, each component of which is a \textsc{diff} ray, we further define a \textsc{diff} regular neighborhood to be a \textsc{diff} codimension~$0$ submanifold that is a disjoint union of regular neighborhoods of the component rays of $L$.\ss
	
Ambient \textsc{diff} isotopy uniqueness of tubes and collars readily establishes ambient \textsc{diff} isotopy uniqueness of such \textsc{diff} regular neighborhoods. With some care, the isotopy can be kept fixed outside any open neighborhood of the union of two such regular neighborhoods.\ss
	
Observe that this definition makes it easy to see that $\mathbb{R}^{m}_{+}$ is a \textsc{diff} regular neighborhood of any affine linear ray in $\mathbb{R}^{m}$ that lies in the interior of $\mathbb{R}^{m}_{+}$. Indeed, the corresponding tube about $r$ can have spherical caps as fibers as shown in Figure~\ref{linrayrn}.
Note that Proposition~\ref{propC} above is the \textsc{pl} analogue of this fact.
\begin{figure}[h!]
	\centerline{\includegraphics{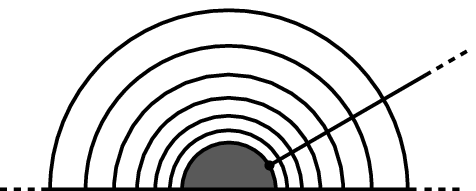}}
	\captionsetup{labelformat=ss,textfont={it,small},labelsep=newline,justification=centerlast,width=\textwidth}
	\caption{$\mathbb{R}^{m}_{+}$ is a \textsc{diff} regular neighborhood of an affine linear ray.}
	\label{linrayrn}
\end{figure}

\nss{TOP Open Regular Neighborhoods}
	
There is no simple elementary theory of \emph{closed} \textsc{top} regular neighborhoods. This deficiency will be overcome using a simple elementary notion of \emph{open regular neighborhood} that is adequate for proving the Cantrell-Stallings hyperplane unknotting theorem for \textsc{top} using \textnormal{CSI}. Incidentally, such open regular neighborhoods could serve in proving the \textsc{pl} and \textsc{diff} versions of the hyperplane unknotting theorem, in lieu of the more precise closed \textsc{cat} regular neighborhood theory.\ss
	
Let $W$, $X$, $Y$ and $Z$ be locally compact (but not necessarily compact!) metrizable spaces, where $Z$ is a closed subset of $W$. Consider a proper continuous surjection $f:X \to Z$ and define the \dfb{infinite radius mapping cylinder} $\text{Map}(f)$ to be the quotient of the disjoint union $X \times [0,\infty) \sqcup Z$ by the relation that identifies $(x,0)$ to $f(x) \in Z$ for all $x \in X$. Clearly, $Z$ is closed in $\text{Map}(f)$ and the open subset $X\times (0,\infty)$ is its complement. For $\rho > 0$, we define the \dfb{radius} $\rho$ \dfb{mapping cylinder} $\text{Map}_{\rho}(f)$ to be the quotient of $X \times [0,\rho] \sqcup Z$ in $\text{Map}(f)$ and also the open one $\text{Map}_{<\rho}(f)$ to be the quotient of $X \times [0,\rho)\sqcup Z$ in $\text{Map}(f)$.\ss
	
Let $g:Y \to Z$ be another such map (same target, but different source). Suppose that $\text{Map}(f)$ and $\text{Map}(g)$ are embedded, fixing $Z$, as open neighborhoods of $Z$ in $W$. Then, we have the following well known result, where $A \Subset B$ for sets in a space $W$ means that the closure of $A$ in $W$ is contained in the interior of $B$ in $W$.
	
\begin{theorem}[Open Mapping Cylinder Neighborhood Uniqueness]\label{t:omcnu}
If $\textnormal{Map}_1(f) \Subset \textnormal{Map}(g)$, then there exists a homeomorphism of $\textnormal{Map}(f)$ onto $\textnormal{Map}(g)$ that fixes pointwise $\textnormal{Map}_{1}(f)$. Consequently \[\textnormal{Map}(g) - \textnormal{Map}_{<1}(f)\] is homeomorphic to \mbox{$X\times [1,\infty)$} fixing $X\times 1$.
\end{theorem}
	
\begin{remarks}
\Item(1)! Although $\textnormal{Map}_1(f)$ is clearly closed in $\textnormal{Map}(f)$, the conclusion of the theorem is false if $\textnormal{Map}_1(f)$ is not closed in $\textnormal{Map}(g)$, and this may occur even when $\textnormal{Map}(f)\subset\textnormal{Map}(g)$ as shown in Figure~\ref{closedCE}.
\begin{figure}[h!]
	\centerline{\includegraphics{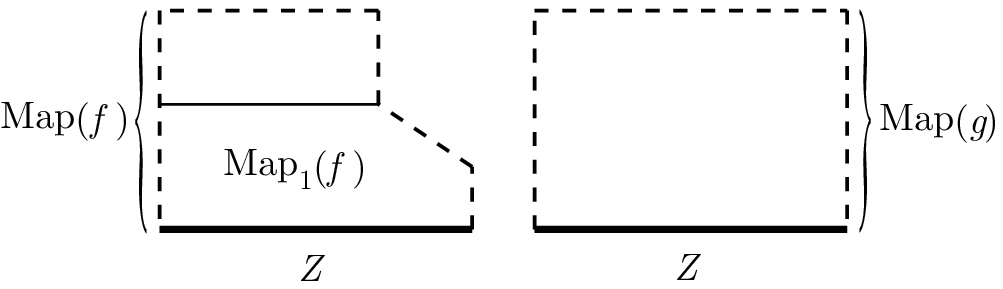}}
	\captionsetup{labelformat=ss,textfont={it,small},labelsep=newline,justification=centerlast,width=.88\textwidth}
	\caption{Example where $\textnormal{Map}(f)\subset\textnormal{Map}(g)$, but $\textnormal{Map}_1(f)$ is not closed in $\textnormal{Map}(g)$ and hence is not closed in any space $W$ containing $\textnormal{Map}(g)$. Here, $\textnormal{Map}(g)={(0,1)\times[0,\infty)}$, $Z={(0,1)\times0}$, and ${X=(0,1)}$.}
	\label{closedCE}
\end{figure}
On the other hand, there then always exists a self-homeomorphism $h$ of $\textnormal{Map}(f)$ such that $h(\textnormal{Map}_1(f))$ is closed in $W$.

\Item(2)! Even when $X$ and $Y$ are both \textsc{top} manifolds, the conclusion of this theorem does not imply that $X$ is homeomorphic to $Y$. Further, if they happen to be homeomorphic, $X \times 1$ is not in general ambient isotopic to $Y \times 1$ (see~\cite{milnor61} and the \textsc{top} invariance of simple homotopy type in~\cite{kirbysiebenmann69}, and Essay III of~\cite{kirbysiebenmannbook}).

\Item(3)! Theorem~\ref{t:omcnu} remains true if $X$, $Y$, and $Z$ are merely Hausdorff and paracompact~\cite[p.~260]{siebenmann73}, but we do not need this generality.
\end{remarks}

\begin{noqedproof}
The most appropriate proof to recall here is one using an infinite composition trick that is often called the {\textquoteleft}Eilenberg-Mazur swindle{\textquoteright} (see also~\cite{stallings62groups,stallings65}). Without loss we may assume that $W =\textnormal{Map}(g)$. After reembedding $\text{Map}(f)$ and $\text{Map}(g)$ into $W$ by suitable topological automorphisms of $\text{Map}(f)$ and $\text{Map}(g)$ respectively, with their supports disjoint from $\textnormal{Map}_{1}(f)$, we can assume that radius 1 and radius 2 mapping cylinders are shuffled as follows
\begin{equation}\label{e:mapshuffle}
\text{Map}_{1}(f) \Subset \text{Map}_{1}(g) \Subset \text{Map}_{2}(f) \Subset \text{Map}_{2}(g). \tag{\textasteriskcentered}
\end{equation}
Here one uses the local compactness and metrizability hypotheses (see~\cite{kwunraymond}).\ss

The triad $\gamma = (V;X \times 1,Y \times 1)$, where $V$ is $\text{Map}_{1}(g)$ minus the topological interior of $\text{Map}_{1}(f)$, can be regarded as a cobordism\footnote{In this context, cobordism means that the two subspaces of each triad are identified in the obvious way to $X$ or to $Y$. Cobordism isomorphism (indicated by $\cong$) means triad homeomorphism respecting these identifications.} from $X$ to $Y$ (see~\cite{milnor65}).
\begin{figure}[h!]
	\centerline{\includegraphics{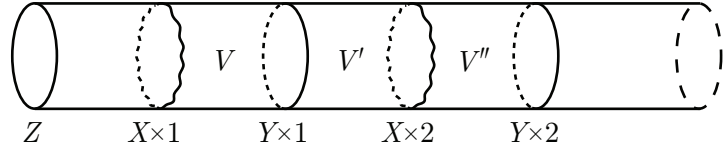}}
	\captionsetup{labelformat=ss,textfont=it}
	\caption{Shuffled mapping cylinders.}
	\label{mapcyl}
\end{figure}
The relations \eqref{e:mapshuffle} show that $\gamma$ has an inverse $\gamma' = (V';Y \times 1,X \times 2)$ viewed as a cobordism from $Y$ to $X$, where $V'$ is $\text{Map}_{2}(f)$ minus the topological interior of $\text{Map}_{1}(g)$ (see Figure~\ref{mapcyl}). In other words, the end to end cobordism composition $\gamma \cdot \gamma'$ is topologically the product cobordism $\varepsilon_{X}$ on $X$, written $\gamma \cdot \gamma' \cong \varepsilon_{X}$. Similarly, $\gamma'$ has an inverse $\gamma''=(V'';X\times 2,Y\times 2)$, where $V''$ is $\text{Map}_{2}(g)$ minus the topological interior of $\text{Map}_{2}(f)$, written $\gamma' \cdot \gamma''\cong \varepsilon_{Y}$. Using an obvious associativity, we see that $\gamma$ and $\gamma''$ are isomorphic cobordisms
\[
	\gamma 	\cong \gamma \cdot \varepsilon_{Y} \cong \gamma \cdot (\gamma' \cdot \gamma'') 
					\cong (\gamma \cdot \gamma') \cdot \gamma'' \cong \varepsilon_{X} \cdot \gamma'' \cong \gamma''.
\]
In particular, $\gamma' \cdot \gamma \cong \varepsilon_{Y}$.\ss
	
$\text{Map}(f)$ minus the interior of $\text{Map}_{1}(f)$ is (the body of) the infinite cobordism composition 
$\varepsilon_{X} \cdot \varepsilon_{X} \cdot \varepsilon_{X} \cdots$, while $\text{Map}(g)$ minus the interior of $\text{Map}_{1}(f)$ is the infinite composition $\gamma \cdot \varepsilon_{Y} \cdot \varepsilon_{Y} \cdot \varepsilon_{Y} \cdots$.
But, these are the same by the infinite product swindle, again using associativity
\begin{align}
\gamma \cdot \varepsilon_{Y} \cdot \varepsilon_{Y} \cdot \varepsilon_{Y} \cdots
	&\cong \gamma \cdot (\gamma' \cdot \gamma) \cdot (\gamma' \cdot \gamma) \cdot
		(\gamma' \cdot \gamma) \cdots \notag \\
	&\cong (\gamma \cdot \gamma') \cdot (\gamma \cdot \gamma') \cdot
		(\gamma \cdot \gamma') \cdots \notag \\
	&\cong \varepsilon_{X} \cdot \varepsilon_{X} \cdot \varepsilon_{X} \cdots. \notag \hspace{126 pt} \qed
\end{align}	
\end{noqedproof}

\section{Radial Ray and Linear Gasket Uniqueness}\label{s:raysgaskets}

In this section, \textsc{cat} will mean either \textsc{pl} or \textsc{diff}. We begin with a ray unknotting lemma for radial rays in $\mathbb{H}^{m}$. Let $L \cong \mathbb{Z}_{+} \times [0,\infty)$ be a proper \textsc{cat} embedded submanifold of $\mathbb{H}^{m}$ so that all rays $r_{i} = i \times [0,\infty)$ are \dfb{radial}, i.e. each ray is contained in a line through the origin in $\mathbb{R}^{m} \supset \mathbb{H}^{m}$, and is disjoint from the origin. In what follows, lengths come from the standard euclidean metric on $\mathbb{R}^{m}$. For each $i \in \mathbb{Z}_{+}$, let $d_{i}$ denote the distance from the origin in $\mathbb{R}^{m}$ to the initial point of $r_{i}$ parameterized by $i\times 0$, and let $p_{i}$ denote the limit point of $r_{i}$ in $S^{m-1} = {\partial} \overline{ \mathbb{H}}^{m}$, where 
$\overline{\mathbb{H}}^{m}$ is the unit ball $B^{m}$ that is the closure of $\mathbb{H}^{m}$ in $\mathbb{R}^{m}$.\ss

Let $L'$ be another such submanifold, and define $r'_{i}$, $d'_{i}$, and $p'_{i}$ in the same way. Also, let $f: L \to L'$ be a \textsc{cat} isomorphism. Notice that both sequences $d_{i}$ and $d'_{i}$ converge to 1 as $i \to \infty$ since $L$ and $L'$ are properly embedded.\ss

$S(t_{1},t_{2})$, with $0 < t_{1} < t_{2} \leq 1$, will denote the thickened sphere of points $x \in \mathbb{H}^{m}$ such that $t_{1} \leq \| x \| \leq t_{2}$.

\begin{lemma}[Radial Ray Uniqueness]\label{l:radrays}
With the above data, suppose $m\geq 3$. Then there is a \textsc{cat} ambient isotopy $h_{t}$ of $\mathbb{H}^{m}$, $0 \leq t \leq 1$, such that $h_{0}$ is the identity and $h_{1} |_{L} = f$.
\end{lemma}

\begin{remark}
This ambient isotopy of $\mathbb{H}^{m}$ cannot in general extend to an ambient isotopy of the ball $\overline{\mathbb{H}}^{m}$ since the accumulation points in ${\partial}\overline{\mathbb{H}}^{m}$ of $L$ would then be homeomorphic to those of $L'$. On the other hand, $p_{i}$ can be an arbitrary sequence of distinct points in ${\partial}\overline{\mathbb{H}}^{m}$; thus its set of accumulation points in ${\partial}\overline{\mathbb{H}}^{m}$ may be any nonempty compact subset.
\end{remark}

\begin{diffpfradrays}
Reindex the rays $r_{i}$ so that $f(r_{i}) = r'_{i}$. Since any \textsc{diff} automorphism of $[0,\infty)$ is isotopic to the identity, it will suffice to construct $h_{t}$ as above so that $h_{1}(r_{i}) = r'_{i}$. Reindexing rays, we can assume that $d_{i} \leq d_{i+1}$ for $i \in \mathbb{Z}_{+}$. It is elementary that $d_{i} \to 1$ as $i \to \infty$.\ss

A preliminary ambient isotopy sets the stage. Shrink the rays $r_{i}$ radially towards their limit points $p_{i}$ so that $d_{i} \geq d'_{i}$ (while maintaining $d_{i} \leq d_{i+1}$). This is straightforward using a regular neighborhood of $L$ in $\mathbb{H}^{m}$.\ss

Now, choose a \textsc{diff} simple path $\alpha_{1}$ in $S^{m-1}$ from $p_{1}$ to $p'_{1}$. This path obviously permits construction of an isotopy of $S^{m-1}$ supported near the path and taking $p_{1}$ to $p'_{1}$. Extending this isotopy radially gives an ambient isotopy of $S(d_1,1)$ taking the ray $r_{1}$ to a subset of $r'_{1}$ (recall, we arranged that $d_{i} \geq d'_{i}$). This ambient isotopy of $S(d_1,1)$ extends naturally to one of $\overline{\mathbb{H}}^{m}$ fixing the ball of radius $d_{1} - \varepsilon_{1}$ for any small $\varepsilon_{1} > 0$. At the end of this isotopy, any ray $r_{i}$, $i \geq 2$, that moved has image another radial ray of the same length which (abusing language) we still refer to as $r_{i}$ with endpoint $p_{i}$.\ss

Next, similarly form an isotopy of $S^{m-1}$ moving $p_{2}$ to $p'_{2}$ and having support missing $p'_{1}$. Extending radially to $S(d_{2},1)$ we get an ambient isotopy (fixing $r'_{1} \supset r_{1}$) taking $r_{2}$ to a subset of $r'_{2}$.\ss

Inductively form an isotopy of $S^{m-1}$ moving $p_{i}$ to $p'_{i}$ and with support disjoint from $p'_{j}$, for $1 \leq j \leq i-1$. Extend as before to get an ambient isotopy of $\overline{\mathbb{H}}^{m}$ with support in $S(d_{i}-\varepsilon_{i},1)$ taking $r_{i}$ to a subset of $r'_{i}$ while fixing $r'_{j} \supset r_{j}$, for $1 \leq j \leq i-1$. Here, $\varepsilon_{i}$ lies in $(0,d_{i})$ and $\varepsilon_{i}\to 0$ as $i \to \infty$. Also, since $d_{i} \to 1$ as $i \to \infty$, the points in any compact set in $\mathbb{H}^{m}$ are moved at most finitely many times. Hence, the time interval composition of all of these ambient isotopies provides a well defined ambient isotopy of $\mathbb{H}^{m}$ (but usually  not one of $\overline{\mathbb{H}}^{m}$). We now have $r_{i} \subset r'_{i}$ for all $i \in \mathbb{Z}_{+}$. A final ambient isotopy stretches each $r_{i}$ so that $d_{i} = d'_{i}$, finishing the proof for \textsc{diff}. \qed
\end{diffpfradrays}

\begin{plpfradrays}
Make a preliminary \textsc{pl} identification $\Theta$ of the thickened standard \textsc{pl} $(m-1)$-sphere $\Sigma^{m-1} \times (0,1)$ to the complement of the origin $\mathbb{H}^{m}-0$ in such a way that each component of $L$ and of $L'$ lies in a modified ray
\[
	\Theta\left((\textnormal{point})\times(0,1)\right)\subset \Theta\left(\Sigma^{m-1} \times (0,1)\right)=\mathbb{H}^{m}-0.
\]
Now, imitate the \textsc{diff} proof. \qed
\end{plpfradrays}

\begin{remark}
There is no such \textsc{pl} identification $\Theta$ that sends \emph{every} ray of the form $\left((\textnormal{point})\times(0,1)\right)\subset\Sigma^{m-1}\times(0,1)$ to a radial ray in $\mathbb{H}^{m}-0$,  not even when $m=2$. This is a corollary of the observation that the point preimages under any linear surjection $\mathbb{R}^{m} \to \mathbb{R}^{m-1}$ are the set of all lines in $\mathbb{R}^{m}$ \emph{parallel} to the kernel line. Thus, the construction of $\Theta$ must be adapted to $L$ and $L'$, for example by using a well chosen triangulation in which $L$ and $L'$ are $1$-subcomplexes.
\end{remark}

\bigskip

Combined with \textsc{pl} and \textsc{diff} regular neighborhood theory (see Section~\ref{s:regnhbds}), the above radial ray uniqueness lemma (Lemma~\ref{l:radrays}) will let us prove a linear gasket uniqueness lemma that we now formulate. Adopting the context and terminology established for Lemma~\ref{l:radrays}, fix a category \textsc{cat} to be \textsc{pl} or \textsc{diff}. Let $G$ be a linear gasket of dimension $m \geq 3$; i.e., a submanifold of $\mathbb{H}^{m}$ bounded by countably many disjoint hyperbolic hyperplanes $H_{i}$, $i \in \mathbb{Z}_{+}$. Let $G'$ be another such gasket of dimension $m$ and distinguish corresponding subsets by primes.

\begin{lemma}[Linear Gasket Uniqueness]\label{l:gasketuniqueness}
Given the data above, there is a \textsc{cat} ambient isotopy $g_{t}$ of $\mathbb{H}^{m}$, $0 \leq t \leq 1$, so that $g_{0}=\textnormal{id}|_{\mathbb{H}^{m}}$, $g_{1}(G) = G'$, and $g_{1}(H_{i}) = H'_{i}$ for all $i \in \mathbb{Z}_{+}$.
\end{lemma}

\begin{corollary}\label{c:gasketboundaryextension}
If $G$ and $G'$ are gaskets and $f: {\partial}G \to {\partial}G'$ is a degree $+1$ \textsc{cat} isomorphism of their boundaries, then $f$ extends to a \textsc{cat} isomorphism $F: G \to G'$. \qed
\end{corollary}

\begin{gasketuniquenessproof}
Without loss of generality, we assume that $G$ and $G'$ are linear gaskets in $\mathbb{H}^{m}$. The gasket $G$ determines a canonical \textsc{cat} submanifold $L \cong \mathbb{Z}_{+} \times [0,\infty)$ of $\mathbb{H}^{m}$ as follows: each hyperplane boundary component $H_{i}$, $i \in \mathbb{Z}_{+}$, of the gasket $G$ defines a proper radial ray $r_{i}$ in $\mathbb{H}^{m}$, namely the one orthogonal to $H_{i}$, having endpoint the point of $H_{i}$ closest to the origin in $\mathbb{R}^{m}$, and extending outwards from $G$. The union of these rays is defined to be $L$. For each $H_{i}$, let $V_{i}$ denote the closed complementary component of $\mathbb{H}^{m} - \text{Int}G$ with boundary $H_{i}$. If $r$ is any radial ray in $\mathbb{H}^{m}$, then let $s$ denote the radial ray obtained from $r$ by shrinking it outwards radially to be half as long (for the euclidean metric). Each $V_{i}$ is isomorphic to the closed upper half space $\mathbb{R}^{m}_{+}$ and is a \textsc{cat} regular neighborhood of $s_{i}$ as we have observed in Section~\ref{s:regnhbds}. Similarly, the gasket $G'$ canonically determines closed complementary components $V'_{i}$, rays $r'_{i}$, and shortened rays $s'_{i}$.\ss

After a preliminary isotopy provided by Lemma~\ref{l:radrays} (radial ray\break uniqueness), we may assume $r_{i} = r'_{i}$ for all $i \in \mathbb{Z}_{+}$, and hence $s_{i} = s'_{i}$ for all $i \in \mathbb{Z}_{+}$. By \textsc{cat} regular neighborhood ambient uniqueness, we may now ambiently isotop $V_{i}$ to $V'_{i}$ for all $i$ (simultaneously), completing the proof. \qed
\end{gasketuniquenessproof}

\begin{remark}
Lemmas~\ref{l:radrays} and~\ref{l:gasketuniqueness} hold also for a finite index set in place of $\mathbb{Z}_{+}$. One can deduce this from the case of $\mathbb{Z}_{+}$. Or, one can note that the same proofs apply.
\end{remark}

\begin{remark}\label{radrays2d}
As it is stated, Lemma~\ref{l:radrays} (radial ray uniqueness) fails in dimension $2$, even for three rays. Any set of distinct radial rays in $\mathbb{H}^2$ obviously inherits a natural cyclic order from that of their limit points on the circle $S^1$. The proof of Lemma~\ref{l:radrays} actually shows that such a collection of rays is determined up to ambient isotopy of $\mathbb{H}^{2}$ by the isomorphism class of its cyclic ordering. There are many such classes when the number of rays is infinite. For example, the number of rays with no immediate successor (or predecessor) is then an invariant. In fact, there are uncountably many such classes. We are confident that, taking account of this natural ray order, one can nevertheless define an associative \textnormal{CSI} operation for 2-manifolds. It is non-commutative in general for 2-manifolds with boundary.
\end{remark}

\section{Proof of Theorem~\ref{p:csiproperties}: Basic Properties of \CSI/}\label{s:proofbasiccsiprops}

\nss{Proof of Property~(1):\hfill\break Well Definition and Commutativity of \CSI/}

Let \textsc{cat} be \textsc{pl} or \textsc{diff}. Recall that with the data introduced for the statement of Property (1) of Theorem~\ref{p:csiproperties} above, we are seeking a certain sort of \textsc{cat} isomorphism of triples $\psi :(W,G,Q) \to (W',G',Q')$. On the closed complements of the gasket interiors, this $\psi$ is rigidly prescribed by the data; call this $\psi_{0}: W-\text{Int}G \to W'-\text{Int}G'$. This $\psi_{0}$ has degree $+1$ as a map ${\partial}G \to {\partial}G'$. Further, the $\psi$ we seek is prescribed up to isotopy on $Q$ as a degree $+1$ isomorphism $Q \to Q'$. Thus, denoting by $H$ and $H'$ the fine gaskets $G - \text{Int}Q$ and $G' - \text{Int}Q'$, it suffices to extend $\psi|: {\partial}H \to {\partial}H'$ to a \textsc{cat} degree $+1$ isomorphism $H \to H'$ of the fine gaskets. This extension exists by Corollary \ref{c:gasketboundaryextension}. \qed \ss

We explain the notion of a \dfb{countable indexed set} $\mathcal{M}$ of flanged manifolds. It consists of a set $I$ that is finite or countably infinite, and a map of $I$ into the class of flanged $m$-manifolds. The set $I$ is called the \dfb{index set} and, in what follows, will always be a subset of $\mathbb{N}$ or of $\mathbb{N}^2$. If we write $\mathcal{M} = \left\{M_i \mid i\in I\right\}$, then the 
flanged manifold corresponding to $i \in I$ is $M_i$.
It is not always required that $M_i$ and $M_j$ be disjoint or even distinct when $i \neq j$ in $I$. Thus one can also similarly define an indexed set in any class --- in place of the class of flanged manifolds --- for example in the class of \textsc{cat} isomorphism classes of flanged manifolds.
\vfill\newpage

\nss{Composite CSI Operations and Associativity}

To elucidate associativity, we must make its meaning more precise. Our proof of the Cantrell-Stallings hyperplane unknotting theorem uses only a simple (but infinite) associativity which is expressible in traditional algebraic notation. But, the \textnormal{CSI} operation enjoys a natural associativity that is at once more general and equally straightforward to establish. Some tree combinatorics will be involved. More specifically, we introduce what we call a \dfb{tree of flanged gaskets}. The category \textsc{cat} in which we work here is again \textsc{pl} or \textsc{diff}, and the manifold dimension $m$ will be $\geq 3$.\ss 

A \dfb{rooted tree} will mean a countable simplicial
tree (not necessarily locally finite) that has a
distinguished vertex $v_0$ called the \dfb{root}. In such a
tree, there is a natural orientation of the edges.
Indeed, from each vertex $v\neq v_0$ there is
a unique oriented edge $vv'$ joining $v$ to a vertex $v'$
strictly nearer to the root vertex in the obvious 
simplicial path metric.\ss

A \dfb{tree of $m$-dimensional flanged gaskets} is a
rooted abstract tree $\mathcal{G}$ whose vertices and edges are
given as as follows:
\smallskip

\Item(1)! The vertex set of $\mathcal{G}$ is a finite or
countable indexed set $\left\{G_i \mid i\in I\right\}$ of disjoint
$m$-dimensional flanged gaskets. The flange of $G_i$ is
denoted $F_i$ and the root vertex of $\mathcal{G}$ is
denoted $G_0$. Furthermore, the boundary components of
$G_i$ are indexed as $H_{i,j}$, $j\in J_i$.

\Item(2)! There is a unique oriented edge of $\mathcal{G}$
joining any vertex $G_i\neq G_0$ to the
unique adjacent vertex (= flanged gasket) $G_{i'}$ that
is nearer to $G_0$. This edge is presented as an
ordered  pair $\left(G_i,H_{i',j}\right)$ where, as the notation
indicates, $H_{i',j}$ is one of the indexed boundary
components of $G_{i'}$. Each boundary component of the disjoint sum $\left|G\right|=\sqcup_i \left|G_i\right|$ is required to occur in \emph{at most}
one edge of $\mathcal{G}$.
\medskip

By the following gluing process, $\mathcal{G}$ determines a \textsc{cat} \dfb{composite flanged gasket} denoted $\fg{\mathcal{G}}$. In the disjoint sum $\left|G\right|=\sqcup_i \left|G_i\right|$, make these identifications: for
each edge $\left(G_i,H_{i',j}\right)$ of $\mathcal{G}$, identify the flange $F_i$ of $G_i$ to a
small open collar of $H_{i',j}$ in $\left|G_{i'}\right|$ by an
orientation preserving \textsc{cat} isomorphism
$\theta_{i,i'}$. Here `small' should mean inside a prescribed open collar neighborhood of the boundary 
$\partial{(G_{i'} - F_{i'})}$ of the fine gasket of $G_{i'}$, so that the flanges identified into $G_{i'}$ obviously do not intersect. Since degree determines $\theta_{i,i'}$
up to isotopy, $\fg{\mathcal{G}}$ is determined up to \textsc{cat}
isomorphism that is the identity outside of an
arbitrarily small bicollar neighborhood in $\fg{\mathcal{G}}$
of the identified boundary components $\partial{F_i}=H_{i',j}$.

\begin{lemma}\label{cfg}
With the above definitions, the pair $(\fg{\mathcal{G}},F_0)$ is a flanged gasket.
\end{lemma}

The proof of Lemma~\ref{cfg} will come after we complete the definition
of a composite \textnormal{CSI} operation based on $\mathcal{G}$.\ss

For each $i\in I$, consider the set $J^{+}_{i} \subset J_i$ of those $j\in J_i$ (if any) such that, for no $k\in I$, does there exist an edge $\left(G_k,H_{i,j}\right)$. By the construction of $\fg{\mathcal{G}}$, its boundary $\partial{\fg{\mathcal{G}}}$ is a disjoint sum
\[
	\bigsqcup \left\{H_{i,j} \mid i\in I \textnormal{ and } j \in J^+_i \right\}.
\]

By definition, a \dfb{composite CSI operation}
according to the rooted tree $\mathcal{G}$ of flanged
gaskets as above involves an indexed set of flanged
$m$-manifolds to be `summed'
\[
	\left\{M_{i,j} \mid  i \in I \textnormal{ and } j\in J^{+}_i\right\}.
\]
The corresponding `sum' is the flanged manifold (flanged by $F_0$) 
obtained by gluing the flange of each such $M_{i,j}$
by a degree +1 isomorphism to a small open collar of
$H_{i,j}$ in $\fg{\mathcal{G}}$ (clearly this open collar may be chosen in $\left|G_i\right|$).\ss

\begin{pfoflemmacfg}
With the notations established
above, it suffices to prove that the flanged
manifold $\fg{\mathcal{G}}$ is a flanged gasket. This is
immediate from the following more primitive lemma (which will be reused in Section~\ref{s:multiplehyperplanes} in our proof of the MHLT (Theorem~\ref{docilethm})). \qed
\end{pfoflemmacfg}

\begin{lemma}\label{l:TGisgasket}
For \textsc{cat}=\textsc{diff} or \textsc{pl}, suppose that an oriented \textsc{cat}
$m$-manifold $X$ is a finite or countable union of \textsc{cat}
gaskets $G_i$, $i \in I$, any two of which are either
disjoint or intersect in a single boundary
component of each. Suppose also that the nerve of
the closed cover $\left\{G_i \mid i \in I\right\}$ of X is a
simplicial tree $\mathcal{T}$. Then $X$ is a \textsc{cat} gasket.
\end{lemma}

\begin{pfoflemmaTG}
Without loss of generality, we assume $I$ is $\mathbb{N}$ 
or a finite initial segment of $\mathbb{N}$.
Reindexing the $G_i$, we can arrange that,
for all $i \geq 0$, the gasket $G_{i+1}$ is adjacent
in $X$ to the connected block 
\[
	X_i := G_0 \cup G_1\cup \cdots \cup G_i.
\]
By definition, $G_0 = X_0$ can be degree +1
embedded in $\mathbb{H}^{m}$ with frontier made up of
hyperbolic hyperplanes.\ss

Suppose inductively that $\phi_i : X_i \to \mathbb{H}^m$ is
such an embedding for some $i\geq0$. Write $X'_i$ for
$\phi_i(X_i)$, write $H_i$ for the boundary component of
$X_i$ that is shared with $G_{i+1}$, and write $H'_i$ for the hyperbolic hyperplane $\phi_i(H_i)$. We will extend
this embedding $\phi_i$ to one of $X_{i+1} = X_i \cup
G_{i+1}$.\ss

Let $Y^{+}_i$ be the closed halfspace in $\mathbb{H}^{m}$ bounded by
$H'_i$ that does not intersect $\Int{X'_i}$. In $\Int{Y^+_i}$ choose as many disjoint halfspaces (each bounded by a hyperbolic hyperplane)
as $G_{i+1}$ has boundary components disjoint from $X_i$;
then delete the interiors of those halfspaces
from $Y^+_i$. With the intent to assure that the
ultimate embedding of $X$ will be proper, we can and
do
\smallskip
\begin{itemize}
\item[$(\ast)$] choose these halfspaces within the $1/(i+1)$ neighborhood of the frontier 
sphere $S^{m-1}$ of $\mathbb{H}^{m}$
in $\mathbb{R}^m$ (for the euclidean distance of $\mathbb{R}^m$).
\end{itemize}
\smallskip
The result is a linear gasket $G'_{i+1}$ in $\mathbb{H}^m$
adjacent to $X'_i$, more precisely $X'_i \cap
G'_{i+1} = H'_i$. By Corollary~\ref{c:gasketboundaryextension} concerning \textsc{cat} uniqueness of linear gaskets, there is a \textsc{cat}
isomorphism $G_{i+1}\to G'_{i+1}$ agreeing with
$\phi_i$ on $H_i$ and thus extending $\phi_{i}|_{H_i}$ to a
\textsc{cat} embedding $\psi_{i+1}$ of $G_{i+1}$ onto a linear
gasket in $\mathbb{H}^m$. Then $\phi_i$ and $\psi_{i+1}$
together define an injective \textsc{cat} map $X_{i+1}\to\mathbb{H}^m$ that is clearly proper. For \textsc{cat}=\textsc{pl} this injective map induces a \textsc{pl} isomorphism with its image. For
\textsc{cat}=\textsc{diff} this is likewise true after
modification of $\psi_{i+1}$ on a small collar
neighborhood of $H_i$ in $G_{i+1}$ (see~\cite{milnor65}).\ss

This completes the induction defining $\phi_i$ for
$i\in I$. The inductively imposed condition $(\ast)$
assures that:
\smallskip
\begin{itemize}
\item[$(\ast\ast)$] For each $i>0$, the frontier $\partial{G'_i}$ lies
in the $1/i$ neighborhood of $S^{m-1}$.
\end{itemize}
\smallskip
Hence $G'_i$ either contains the ball about the
origin of euclidean radius
 $1-(1/i)$ or else it lies outside that
ball.  Since the sets $\Int{G'_i}$ are disjoint, 
it follows that, for all large $i$, $G'_i$ lies 
outside the ball of radius $1-(1/i)$.
Since $\mathbb{H}^{m}$ is the open ball of radius 1 in $\mathbb{R}^m$,
we conclude that:
\smallskip
\begin{itemize}
\item[$(\ast{\ast}\ast)$] The sets $G'_i$ converge toward Alexandroff's infinity in $\mathbb{H}^m$.
\end{itemize}
\smallskip
Together, the $\phi_i$ clearly define an injective
\textsc{cat} map $\phi:X \to\mathbb{H}^m$. The condition $(\ast{\ast}\ast)$ proves
that $\phi$ is proper and thus a $\textsc{cat}$ embedding onto a 
linear gasket $X'$. \qed
\end{pfoflemmaTG}

\begin{remark}
In the proof of Lemma~\ref{l:TGisgasket}, if the conditions
$(\ast)$ and $(\ast\ast)$ are not imposed and the tree $\mathcal{T}$
contains an infinitely long embedded path,
then the map $\phi: X \to\mathbb{H}^m$
may not be proper. But the closure of $\phi(X)$
seems always to be a linear gasket.
\end{remark}

\nss{Proof of Property (2): Associativity of CSI Operations}

Here we state explicitly, and prove, the
associativity properties of \CSI/ as promised in Property (2) of 
Theorem~\ref{p:csiproperties}. We then deduce two basic corollaries.\ss

By Lemma~\ref{cfg} and the above definition of composite \CSI/
operation we immediately get:

\begin{theorem}[First Associativity Theorem]\label{fat}
\hfill\break
Any fixed composite \CSI/ operation on a  
finite or countably infinite set of disjoint flanged 
\textsc{cat} $m$-manifold summands, $m\geq 3$, 
is isomorphic to a (normal) \CSI/ sum of the same
flanged manifolds. Thus the flanged manifold
resulting from this composite \CSI/ operation depends
(up to \textsc{cat} isomorphism of flanged
manifolds) only on the disjoint sum of the
flanged manifold summands. \qed
\end{theorem}

This quickly implies the

\begin{theorem}[Second Associativity Theorem]\label{sat}
\hfill\break
Consider a nonempty sequence $M_i$ 
(finite or infinite) 
of disjoint flanged $m$-manifolds, $m\geq 3$, where
each $M_i$ is itself a \CSI/ sum of a
sequence (finite or infinite) of disjoint flanged
manifolds $M_{i,j}$. Then, writing $\mathcal{M}$ 
for the set $\left\{M_i\right\}$ and $\mathcal{M}'$ for the set
$\left\{M_{i,j}\right\}$, there is a \textsc{cat}
isomorphism of flanged \CSI/ sums:
\[
   \CSI/(\mathcal{M}) \cong \CSI/(\mathcal{M}').
\]
\end{theorem}

\begin{proof}
Examine the defining construction for 
$\CSI/(\mathcal{M})$, which uses a flanged gasket with $\left|\mathcal{M}\right|$
boundary components. In it, replace each summand $M_i$
by a copy of  $\CSI/(\mathcal{M}_i)$  where $\mathcal{M}_i :=\sqcup_j \left\{M_{i,j}\right\}$. This reveals
that  $\CSI/(\mathcal{M})$ is isomorphic to a 
composite \CSI/ sum with  summands
$\mathcal{M}'$. Hence, the First Associativity Theorem tells us that
$\CSI/(\mathcal{M}) \cong \CSI/(\mathcal{M}')$.
\end{proof}

\begin{corollary}
Let $\alpha$, $\beta$, and $\gamma$ be flanged $m$-manifolds, $m\geq3$. Then one has a \textsc{cat} isomorphism of flanged manifolds $(\alpha\beta)\gamma\cong\alpha(\beta\gamma)$.
\end{corollary}

This is the usual formulation of associativity for any binary operation. The parentheses in this example and the next serve to indicate order of \CSI/ summation. The expression $(\alpha\beta)$ indicates the flanged manifold for which we have mentioned the alternative notations $\CSI/(\alpha,\beta)$ and $(\alpha \csinf \beta)$.

\begin{noqedproof}
Two applications of the Second Associativity Theorem above
give the two isomorphisms:
\[
	(\alpha \beta) \gamma \cong \alpha \beta \gamma \cong \alpha (\beta \gamma).
	\eqno \qed
\]
\end{noqedproof}

\smallskip

The next corollary will be used in proving the HLT (Theorem~\ref{t:ourunknottingtheorem}).

\begin{corollary}
Let the symbols $a$, $b$, $c$, $\ldots$ of an infinite alphabet stand for \textsc{cat} flanged $m$-manifolds, $m\geq3$. Then one has a \textsc{cat} isomorphism of infinite \CSI/ sums of flanged manifolds:
\begin{equation}\label{assocsums}
	(ab)(cd)(ef)(gh)\cdots \cong a(bc)(de)(fg)\cdots . \tag{\dag}
\end{equation}
\end{corollary}

\begin{noqedproof}
Applying the Second Associativity Theorem to the left hand side of \eqref{assocsums}, one gets the isomorphism of \CSI/ sums:
\[
	(ab)(cd)(ef)(gh)\cdots \cong abcdefgh\cdots.
\]
Similarly,
\[
	a(bc)(de)(fg)\cdots \cong abcdefgh\cdots.
	\eqno \qed
\]
\end{noqedproof}

\nss{Proof of Property~(3): Identity Element}

The easy verifications that the \textnormal{CSI} identity is $\varepsilon = (\mathbb{R}^{m},\mathbb{R}^{m}_{+})$ and $\varepsilon \cong \varepsilon \varepsilon \varepsilon \cdots$ are left to the reader. \qed \ss

\medskip

The proof of the three basic properties of \textnormal{CSI} (Theorem~\ref{p:csiproperties}) is complete. \qed \ss

\section{\textnormal{CSI} Proves the Cantrell-Stallings Hyperplane Unknotting Theorem}\label{s:csicantrellstallings}

The machinery developed thus far suffices to prove the following important hyperplane unknotting theorem~\cite{cantrell,stallings65}. Given a manifold $M$ \textsc{cat} isomorphic to some $\mathbb{R}^{k}$, we say a \textsc{cat} ray $r$ embedded in $M$ is \dfb{unknotted} in $M$ if there is a \textsc{cat} isomorphism $f: \mathbb{R}^{k} \to M$ such that $f^{-1}(r)$ is linear in $\mathbb{R}^{k}$.\ss

\begin{theorem}[Hyperplane Linearization Theorem (=~HLT)]\label{t:ourunknottingtheorem}
\hfill\break
Consider a codimension~$1$ and \textsc{cat} proper submanifold $N$ of $\mathbb{R}^{m}$,\break $m \geq 2$, that is \textsc{cat} isomorphic to $\mathbb{R}^{m-1}$. Assume that there is a ray $r$ in $N$ that is unknotted both in $N$ and in $\mathbb{R}^{m}$. Then, $N$  is itself unknotted in the sense that $g(N)$ is linear for some \textsc{cat} automorphism $g$ of $\mathbb{R}^{m}$.
\end{theorem}

\begin{remark}\label{rhs}
The ray unknotting hypothesis facilitates our \textnormal{CSI} based proof for $m\geq3$. The next section shows it is superfluous if $m > 3$.
\end{remark}

\begin{remark}\label{hs2d}
Dimension $2$ is special in that, not only is the ray unknotting hypothesis unnecessary,  but in the case of \textsc{top} the abiding assumption of local local flatness is redundant by the classical Schoenflies theorem (see~\cite{moise},~\cite{siebenmann05}).
\end{remark}

\begin{pfofourunknottingtheoremmeq2}
This is known by classical methods that are explained in~\cite{siebenmann05}. Some details follow.\ss

Case \textsc{cat}=\textsc{top}. One-point (Alexandroff) compactify the pair\break $\left(\mathbb{R}^{2},N\right)$ to produce a pair $(S^{2},\widehat{N})$. The (difficult) classical Schoenflies theorem tells us $(S^{2},\widehat{N})$ is homeomorphic to the standard pair $\left(S^{2},S^{1}\right)$. From this it follows, on deleting the added point $\infty$, that the pair
$\left(\mathbb{R}^{2},N\right)$ is homeomorphic to $\left(\mathbb{R}^{2},\mathbb{R}^{1}\times0\right)$.\ss

Case \textsc{cat}=\textsc{pl}. Proceed similarly but use the ``Almost \textsc{pl} Schoenflies Theorem'' (=~APLST) of Sections~5 and~7 of~\cite{siebenmann05} (see also Remark~\ref{2dwealth} below) to conclude that $(S^{2},\widehat{N})$ is homeomorphic to $\left(S^{2},S^{1}\right)$ by a homeomorphism that is \textsc{pl} except at $\infty$. On deleting $\infty$ we get the desired \textsc{pl} isomorphism between $\left(\mathbb{R}^{2},N\right)$ and $\left(\mathbb{R}^{2},\mathbb{R}^{1}\times0\right)$.\ss

Case \textsc{cat}=\textsc{diff}. It is possible to imitate the above proof for \textsc{pl}.
Alternatively, embedded Morse theory offers an interesting proof that is described in Remarks~\ref{2mrtremarks} following the proof of the $2$-dimensional Multiray Radialization Theorem (=~MRT, Theorem~\ref{mrt}) in Section~\ref{s:multiplehyperplanes}. \qed
\end{pfofourunknottingtheoremmeq2}

\begin{remark}[on overlapping $2$-dimensional results and techniques]\label{2dwealth}
Fortunately, one overlap simplifies: in
dimension $2$ one can usually shift results, at the statement level,
between any two of the three categories \textsc{diff}, \textsc{pl}, and \textsc{top} by
appealing to what can be called \hbox{``$2$-Hauptvermutung''} theorems, for which good references are~\cite{moise}, or Section~9 of~\cite{siebenmann05}.

Another simplification comes from the coincidence of these three
seemingly different properties for connected noncompact surfaces
with all boundary components noncompact: irreducibility, planarity,
and contractibility. A proof will be given as Proposition~\ref{equivcond}.

On the other hand, in dimension $2$, there is a somewhat
confusing wealth of techniques and names for them. We
now illustrate for the present article.

What is called the ``Irreducible \textsc{pl} Surface Classification
Theorem'' (=~PLCT) in Section~7 of~\cite{siebenmann05} is a direct \textsc{pl}
classification, using the very simple \textsc{pl} Schoenflies theorem, for
all \textsc{pl} connected noncompact planar surfaces with finitely many boundary
components all noncompact. This PLCT was used in~\cite{siebenmann05} to prove
the classical Schoenflies theorems that are used in the proofs of
$2$-HLT just given. Also, PLCT clearly \emph{directly} implies the \textsc{pl} case
of $2$-HLT.

Serious overlap of techniques is going to appear when we attack the \hbox{$2$-dimensional}
 Multiple Hyperplane Linearization Theorem\break\hbox{(=~$2$-MHLT)}
in Section~\ref{s:multiplehyperplanes}. The PLCT just mentioned will turn out to be
synonymous with the case for \emph{finitely many} boundary
components of the $2$-dimensional ``Gasket Recognition Theorem''
(=$2$-GRT), see Corollary~\ref{grt} below; this $2$-GRT generalizes PLCT in
that it allows an infinite number of boundary components. Toward
the end of Section~\ref{s:multiplehyperplanes}, we will observe that $2$-GRT is equivalent to a
classification of all contractible $2$-manifolds, and we will ultimately give
three amazingly different proofs of it, which respectively focus on
embedded Morse theory, end theory, and hyperbolic geometry.
\end{remark}

\begin{pfofourunknottingtheoremmeq3pldiff}
It suffices to prove that $A$ and $A'$, the closures of the two components of $\mathbb{R}^{m} - N$ in $\mathbb{R}^{m}$, are \textsc{cat} isomorphic to the closed upper half space $\mathbb{R}^{m}_{+} \subset \mathbb{R}^{m}$. From $A$ construct a \textnormal{CSI} pair $\alpha = (A \cup P,P)$ where $P$ is an open collar neighborhood in $A'$ of ${\partial}A' = N = {\partial}A$, the orientation of $P$ being inherited from $\mathbb{R}^{m}$. Similarly, construct $\alpha' = (A' \cup P',P')$.

\begin{assertion}\label{hlta1}
The \textnormal{CSI} composition $\alpha \alpha'$ of $\alpha$ and $\alpha'$ is \textsc{cat} isomorphic to the trivial \textnormal{CSI} pair $\varepsilon = (\mathbb{R}^{m},\mathbb{R}^{m}_{+})$ that is the identity for the \textnormal{CSI} operation.
\end{assertion}

\begin{pfhlta1}
The key idea is to perceive, embedded in $\mathbb{R}^{m}$, the coarse and fine gaskets for the \textnormal{CSI} operation $\alpha\alpha'$ as suggested by Figure~\ref{embeddedgasket}. Its coarse gasket can clearly be a bicollar neighborhood $G$ of $N$ in $\mathbb{R}^{m}$. We shall prove that a fine gasket is
\[G^{\ast} = G - \text{Int}T(r)\]
where $T(r)$ is a regular neighborhood of $r$ in $G$.
\begin{figure}[h!]
	\centerline{\includegraphics[scale=0.75]{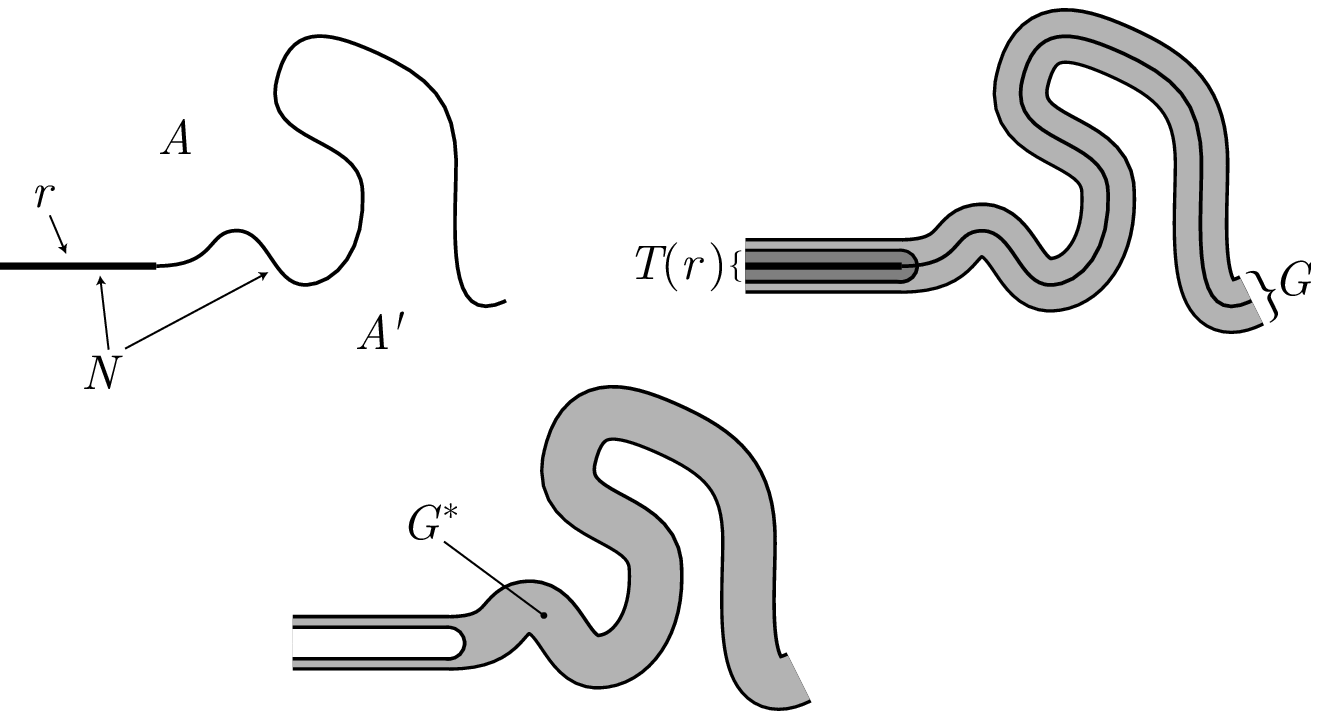}}
	\captionsetup{labelformat=ss,textfont={it,small},labelsep=newline,justification=centerlast,width=\textwidth}
	\caption{Gaskets $G$, $G^{\ast}$, and flange $T(r)$ for a hyperplane unknotting problem.}
	\label{embeddedgasket}
\end{figure}
Since $r$ is, by hypothesis, unknotted in $N$, this $G^{\ast}$ is easily seen to be a gasket; it has three boundary components. Note that the three closed complementary components of $\text{Int}G^{\ast}$ in $\mathbb{R}^{m}$ are respectively isomorphic to $A$, $A'$ and $T(r)$.\ss

Since $T(r)$ is \textsc{cat} degree $+1$ isomorphic to $\mathbb{R}^{m}_{+}$, it is a \textnormal{CSI} flange, and we conclude from the definition of \textnormal{CSI} that the \textnormal{CSI} pair $(\mathbb{R}^{m},T(r))$ is (up to \textnormal{CSI} pair isomorphism) a \textnormal{CSI} product $\alpha \alpha'$, whose coarse and fine gaskets are $G$ and $G^{\ast}$.\ss

Since $r$ is, by hypothesis, unknotted in $\mathbb{R}^{m}$, it follows, by \textsc{pl} and \textsc{diff} ambient regular neighborhood uniqueness (see Section~\ref{s:regnhbds}), that the complement of $\Int \, T(r)$ in $\mathbb{R}^{m}$ is \textsc{cat} isomorphic to $\mathbb{R}^{m}_{+}$. Therefore
\begin{equation}\label{e:pairtrivial}
	(\mathbb{R}^{m},T(r)) \cong (\mathbb{R}^{m},\mathbb{R}^{m}_{+}) = \varepsilon, \tag{\dag}
\end{equation}
where $\cong$ denotes \text{CSI} pair isomorphism.\ss

Taken together, the last two paragraphs prove the assertion that $\alpha \alpha' \cong \varepsilon$. \qed
\end{pfhlta1}

The assertion quickly implies the theorem using the Eilenberg-Mazur swindle. First, $\varepsilon \cong \alpha \alpha' \cong \alpha' \alpha$ using commutativity, so $\alpha$ and $\alpha'$ are mutually inverse. Whence, the infinite product swindle using associativity
\[\alpha \cong \alpha \varepsilon \varepsilon \varepsilon \cdots
  \cong \alpha (\alpha' \alpha) (\alpha' \alpha) \cdots
  \cong (\alpha \alpha') (\alpha \alpha') (\alpha \alpha') \cdots
  \cong \varepsilon \varepsilon \varepsilon \cdots
  \cong \varepsilon.
\]
Also, $\alpha' \cong \alpha' \varepsilon \cong \alpha' \alpha \cong \varepsilon$. Thus, $A$ and $A'$ are \textsc{cat} isomorphic to $\mathbb{R}^{m}_{+}$ as required. This establishes the HLT (Theorem~\ref{t:ourunknottingtheorem}) for $m\geq3$ and \textsc{cat}=\textsc{pl} or \textsc{cat}=\textsc{diff}. \qed
\end{pfofourunknottingtheoremmeq3pldiff}

\medskip

\begin{pfofourunknottingtheoremmeq3top}
Like Cantrell, we will use only elementary arguments. In particular, recall that, by using our refined version of the definition of \textnormal{CSI} for \textsc{top} pairs given at the end of Section~\ref{s:csi}, we have avoided use of the Stable Homeomorphism Theorem (=~SHT) in establishing the basic properties of \textnormal{CSI}.\ss

We now proceed to adapt to \textsc{top} the above proof of the \textsc{diff} version. It adapts routinely except for the two short paragraphs that apply, to the ray $r$ in $\mathbb{R}^{m}$, the uniqueness of \textsc{diff} regular neighborhoods to deduce the \textsc{diff} \textnormal{CSI} pair isomorphism \eqref{e:pairtrivial}. For \textsc{top}, we now establish \eqref{e:pairtrivial} using the \textsc{top} \emph{open} regular neighborhood uniqueness of Section~\ref{s:regnhbds}.\ss

We can and do choose a linear structure on $N$ such that $r$ is a linear ray in $N$. The \textsc{top} bicollar neighborhood $G$ of $N$ was first established by M.~Brown in~\cite{brown62}; a pleasant alternative construction is due to R.~Connelly, see~\cite[Essay~I,~p.~40]{kirbysiebenmannbook}. This $G$ can then be viewed as a \textsc{diff} gasket of which $r$ and $N$ are smooth submanifolds. However, the inclusion of $G$ into $\mathbb{R}^{m}$ is in general not a \textsc{diff} embedding. Let $T(r)$ be a \textsc{diff} regular neighborhood of $r$ in $G$ and let $G^{\ast} = G - \text{Int}T(r)$ be the resulting fine gasket.

\begin{assertion}\label{hlta2}
The closed complement of $\Int \, T(r)$ in $\mathbb{R}^{m}$ is \textsc{top} isomorphic to $\mathbb{R}^{m}_{+}$. Hence, \eqref{e:pairtrivial} holds for \textsc{top} \textnormal{CSI} pair isomorphism.
\end{assertion}

\begin{pfhlta2}
By hypothesis, $r$ is unknotted in $\mathbb{R}^{m}$. Thus, we can now observe that:
\smallskip

\Item(1)! $\mathbb{R}^{m}$ is an open topological mapping cylinder neighborhood of $r$ \hbox{in $\mathbb{R}^{m}$}.

\Item(2)! The \textsc{diff} regular neighborhood $T(r)$ in $G$ is a closed mapping cylinder neighborhood of $r$ in $\mathbb{R}^{m}$ with topological frontier ${\partial}T(r)$ bicollared in $\mathbb{R}^{m}$.
\smallskip

For (1), define $f:\mathbb{R}^{m-1}\to r=[0,\infty)$ by $f(x)=\left\|x\right\|$ which is a proper surjection. The mapping cylinder $\text{Map}(f)$ embeds homeomorphically onto $\mathbb{R}^{m}$ by the quotient map 
$F:\mathbb{R}^{m-1}\times[0,\infty)\to\mathbb{R}^{m}$ that extends $(0,x)\mapsto (0,f(x))$ and maps each hemisphere with center the origin onto the full sphere containing it, crushing (only) the hemisphere boundary onto a single point of $r$ (see Figure~\ref{mapcylembed}).
\begin{figure}[h!]
	\centerline{\includegraphics{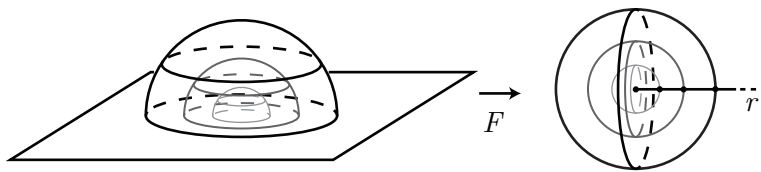}}
	\captionsetup{labelformat=ss,textfont=it}
	\caption{Quotient map $F:\mathbb{R}^{m-1}\times[0,\infty)\to\mathbb{R}^{m}$.}
	\label{mapcylembed}
\end{figure}
Fact (2) follows similarly from our peculiar definition of \textsc{diff} regular neighborhood of a ray (see Section~\ref{s:regnhbds}).\ss

By open mapping cylinder uniqueness (Theorem~\ref{t:omcnu}), these two facts imply that the closed complement in $\mathbb{R}^{m}$ of $\text{Int}\,T(r)$ is \textsc{top} isomorphic to ${\partial}T(r) \times [1,\infty)$. This completes the proof of the assertion. \qed
\end{pfhlta2}

The proof of the HLT (Theorem~\ref{t:ourunknottingtheorem}) for \textsc{top} now concludes as in the \textsc{diff} case. \qed
\end{pfofourunknottingtheoremmeq3top}

\begin{remark}
In the above elementary proof of the \textsc{top} version of the HLT (Theorem~\ref{t:ourunknottingtheorem}), it is \emph{not} proved that the self homeomorphism $g$ of $\mathbb{R}^{m}$ sending $N$ to a linear hyperplane can be chosen ambient isotopic to a linear map. It is always ambient isotopic, but to prove this one needs the SHT of~\cite{kirby69,freedmanquinn}.
\end{remark}

We close this section with some historical remarks on the Cantrell-Stallings theorem.
\medskip

\Item(1)! Progress towards the \textsc{top} theorem from Mazur~\cite{mazurbams} 1959 to Cantrell's full \textsc{top} unknotting theorem in~\cite{cantrell} 1963 was incremental. In 1960, Marston Morse~\cite{morse} extended~\cite{mazurbams} to prove the \textsc{top} version under the extra hypothesis that $N\cup\infty$ is a \textsc{top} bicollared $(m-1)$-sphere in the $m$-sphere $\mathbb{R}^{m}\cup \infty$. Morton Brown's parallel but amazingly novel article~\cite{brown60} 1960 achieved this, too. Then Brown~\cite{brown62} 1962 proved a collaring theorem that replaced the above bicollaring hypothesis by local flatness in the $m$-sphere $\mathbb{R}^{m}\cup \infty$. From 1962 onwards, Cantrell's goal (reached already in 1963) has been viewed as the problem of proving that a codimension $1$ sphere in a sphere of dimension $>3$ cannot have a single `singular' point where local flatness fails.

\Item(2)! Huebsch and Morse~\cite{huebschmorse} 1962 established the \textsc{diff} version under the much stronger unknotting hypothesis that $N$ be linear outside a bounded set in $\mathbb{R}^{m}$.

\Item(3)! Our proof (for any \textsc{cat}) can be viewed as a radical reorganization using \textnormal{CSI} of Cantrell's proof for \textsc{top}~\cite{cantrell}. On the other hand, it was Stallings~\cite{stallings65} who first pointed out the \textsc{diff} version, and formulated a version valid in \emph{all} dimensions. The proof of the \textsc{top} version requires extra precautions (for us, \textsc{diff} gaskets) and extra argumentation (for us, open mapping cylinder neighborhood uniqueness), but, in compensation, it clearly reproves, ab initio, the Schoenflies theorem of Mazur~\cite{mazurbams} and Brown~\cite{brown60,brown62}.

\Item(4)! The apparent novelty, which made us write down the above proof, was our reformulation (circa 2002) of much of the geometry of Cantrell's proof as standard facts about \textnormal{CSI}. This explicit use of some sort of connected sum was, of course, suggested by Mazur's pioneering article~\cite{mazurbams}; compare the `almost \textsc{pl}' version of the Schoenflies theorem in~\cite{rourkesanderson}.

\Item(5)! \textnormal{CSI} itself was not a novelty. Gompf~\cite{gompf} had showed that an infinite \textnormal{CSI} of smooth $4$-manifolds, each homeomorphic to $\mathbb{R}^{4}$, is well defined. He achieved this by proving a multiple ray unknotting result using finger moves; his proof readily extends to all dimensions $\geq4$ (in fact, it is simpler in dimensions $>4$). Gompf used this observation and the infinite product swindle to show that an exotic $\mathbb{R}^{4}$ cannot have an inverse under \textnormal{CSI}. The reader can now check this as an exercise.

\Item(6)! Stallings~\cite{stallings65} deals explicitly only with the \textsc{diff} case. He avoids all connected sum notions. Indeed, the basic entity for which he defines an infinite product operation is a (proper) \textsc{diff} embedding $f: \mathbb{R}^{m-1} \to \mathbb{R}^{m}$ (with an unknotted ray and $m \geq 3$). Stallings' exposition seems to invite formalization in terms of a \emph{pairwise} \textnormal{CSI} operation.

\Item(7)! Johannes de Groot 1972~\cite{degroot} announced a proof of Cantrell's \textsc{top} HLT by generalization of M.~Brown's proof of the \textsc{top} Schoenflies theorem. Regrettably, de Groot died shortly thereafter and no manuscript has surfaced since.

\section{Basic Ray Unknotting in Dimensions >\hspace{0.2em}3}\label{s:rayunknot}

The first goal of this section is to explain the well known fact, mentioned in Remark~\ref{rhs} above, that rays in $\mathbb{R}^{m}$ are related by an ambient isotopy provided that $m > 3$. Then we go on, still assuming $m > 3$, to classify so called multirays in terms of the proper homotopy classes of their component rays.\ss

Throughout this section, \textsc{cat} is one of \textsc{top}, \textsc{pl} or \textsc{diff}. The following basic result will be needed for $1$-manifolds mapping into manifolds of dimension $m > 3$.

\begin{theorem}[Stable Range Embedding Theorem]\label{t:stableembedding}
\hfill\break
Let $f:N^{n} \to M^{m}$ be a proper continuous map of \textsc{cat} manifolds, possibly with boundary. If $2n+1 \leq m$, then $f$ is properly homotopic to a \textsc{cat} embedding $g:N \to M$ such that $g(N)$ lies in $\Int M$. Further, if $2n+2 \leq m$ and $g'$ is a second such embedding properly homotopic to $f$, then there exists a \textsc{cat} ambient isotopy $h_{t}:M \to M$, $0 \leq t \leq 1$, such that $h_{0} = \textnormal{id}|_{M}$ and $h_{1}g=g'$.
\end{theorem}

For \textsc{cat}=\textsc{pl} or \textsc{cat}=\textsc{diff}, the proof is a basic general position argument that can be found in many textbooks. Early references are~\cite{whitney36} and~\cite{bingkister}.\ss

For \textsc{top}, the proof is still surprisingly difficult. One needs a famous method of T.~Homma from 1962~\cite{homma62}, as applied by H.~Gluck~\cite{gluckembbams,gluckembann}. Many expositions of these types of results (in particular~\cite{gluckembann}) are given in a compact relative form, from which one has to deduce the stated noncompact, nonrelative but proper version by a classic argument involving a skeletal induction in the nerve of a suitable covering (see Essay I, Appendix C in~\cite{kirbysiebenmannbook}).\ss

Next, we show that, in some cases of current interest, all rays are properly homotopic.

\begin{lemma}[Simplest Proper Ray Homotopies]\label{l:simplestproperray}
\hfill\break
Let $X$ be locally arcwise connected and locally compact. Suppose $X$ admits a connected closed collar neighborhood $Y\times [0,\infty)$ of Alexandroff infinity. Then any two proper maps $[0,\infty) \to X$ are properly homotopic.
\end{lemma}

\begin{proof}
Any proper map $f: [0,\infty) \to X$ is proper homotopic to one with image in the closed subset $Y \times [0,\infty) \subset X$, so we can and do assume that $X$ is $Y \times [0,\infty)$.\ss

Then, writing $f(0)=(y,t_{0}) \in Y \times [0,\infty)$, it is easy to construct an explicit proper homotopy of $f$ to the proper continuous radial embedding $r_{y}:[0,\infty) \hookrightarrow X = Y \times [0,\infty)$ sending $t \mapsto (y,t)$ for \hbox{all $t$}.\ss

Finally, for any two points $y$ and $y'$ in $Y$, there is a path from $y$ to $y'$ in $Y$ and any such path provides an explicit proper homotopy from $r_{y}$ to the similarly defined radial embedding $r_{y'}$.
\end{proof}

These last two results, when combined with the Cantrell-Stallings theorem as stated in the last section (Theorem~\ref{t:ourunknottingtheorem}), yield the following Hyperplane Linearization Theorem already announced there.
\begin{theorem}\label{maintheorem}
For $m \ne 3$, any \textsc{cat} submanifold $N$ of $\mathbb{R}^{m}$ that is isomorphic to $\mathbb{R}^{m-1}$ is unknotted in the sense that there is a \textsc{cat} automorphism $h$ of $\mathbb{R}^{m}$ such that $h(N) = \mathbb{R}^{m-1} \times 0 \subset \mathbb{R}^{m}$.
\end{theorem}

\begin{remarks}\label{sec7remarks}
\Item(1)! Remember that, by convention, a \textsc{cat} submanifold is a closed subset and is assumed \textsc{cat} locally flat unless the contrary is explicitly stated.

\Item(2)! The case \textsc{cat}=\textsc{top} of Theorem~\ref{maintheorem} is Cantrell's result as he formulated it. Beware that (still today) any completely bootstrapping proof seems to require an exposition of Homma's method.

\Item(3)! It is well known that a proper ray (any \textsc{cat}) may be knotted in $\mathbb{R}^{3}$. Fox and Artin~\cite[Example~1.2]{foxartin} exhibited the first such ray, Alford and Ball~\cite{alfordball} produced infinitely many knot types and conjectured uncountably many exist, and McPherson~\cite{mcpherson} published a proof of this conjecture (earlier, Giffen 1963, Sikkema, Kinoshita, and Lomonaco 1967, and McPherson 1969 had announced proofs~\cite[p.273]{burgesscannon}). The boundary of a closed regular neighborhood of any such knotted ray is a knotted hyperplane in $\mathbb{R}^{3}$. Still, even in this dimension the knot type of any \textsc{cat} hyperplane $N\subset\mathbb{R}^{3}$ is determined by the knot type in $\mathbb{R}^{3}$ of any \textsc{cat} ray $r\subset N$~\cite{sikkema} (see also~\cite{harroldmoise}); in fact, $N$ is ambient isotopic to the boundary of a \textsc{cat} closed regular neighborhood of $r$ in $\mathbb{R}^{3}$~\cite{calcutking2}. Thus, one of the two closed complementary components of $N$ in $\mathbb{R}^3$ is \textsc{cat} isomorphic to $\mathbb{R}_{+}^{3}$.

\Item(4)! Here is an immediate corollary for \textsc{cat} = \textsc{diff} that concerns the still mysterious dimension 4. Suppose that  $N^3 \subset S^4$ is a smoothly embedded $3$-sphere such that the pair $(S^4,N^3)$ is \emph{not} \textsc{diff} isomorphic to $(S^4, S^3)$ and thus is a counterexample to the unsettled \textsc{diff} \hbox{$4$-dimensional} Schoenflies conjecture. Then, nevertheless, for any point $p$ in $N^3$ one has
$(S^4 - p, N^3 - p)\cong (\mathbb{R}^4, \mathbb{R}^3)$.

\Item(5)! We have seen that the Cantrell-Stallings unknotting theorem is closely related to the fact that: \emph{if $\alpha := (M,P)$ is a dimension $m$ \textsc{cat} \textnormal{CSI} pair that has an inverse up to degree $+1$ isomorphism in the commutative semigroup of isomorphism classes of \textsc{cat} \textnormal{CSI} pairs of dimension $m \geq 3$ under \textnormal{CSI} sum, then $(M,P)$ is in the identity class, namely that of
$(\mathbb{R}^m,\mathbb{R}_{+}^m)$}. Thus it is perhaps of interest to ask about other algebraically expressible facts about this semigroup. For example: \emph{is it true that $\alpha \cong \alpha\beta$ always implies that $\alpha \cong \alpha\beta^{\infty}$?} Curiously, this is false for certain $(M,P)$ where $M$ has more than one end, as Figure~\ref{aneqabinfty} indicates.
\begin{figure}[h!]
	\centerline{\includegraphics{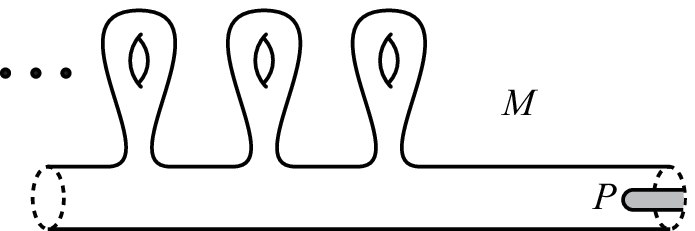}}
	\captionsetup{labelformat=ss,textfont={it,small},labelsep=newline,justification=centerlast,width=\textwidth}
	\caption{\textnormal{CSI} pair $(M,P)$ where $M$ has two ends and one end is collared.}
	\label{aneqabinfty}
\end{figure}
Although this figure is for dimension $2$, it clearly has analogs in all dimensions $>2$. Is this implication true at least when $M$ has one end? Or when $M$ is the interior of a compact manifold?
\end{remarks}

This concludes our exposition of the Cantrell-Stallings theorem.

\section{Singular and Multiple rays}\label{singmultrays}

This section shows that multiple rays embed and unknot much like single rays. We define a \dfb{singular ray} in a locally compact space $X$ to be a proper continuous map $[0,\infty) \to X$. In Section~\ref{s:multiplehyperplanes}, singular rays will be a tool for unknotting multiple hyperplanes in dimensions $>3$.\ss

\begin{lemma}\label{l:rayproper}
Let $f_{i}: [0,\infty) \to X$, with $i$ varying in the finite or countably infinite discrete index set $S$, be singular rays in a locally compact, sigma compact space $X$. Then, for each $i \in S$, one can choose a proper homotopy of $f_{i}$ to a singular ray $f'_{i}$ such that the rule $(i,x) \mapsto f'_{i}(x)$ defines a proper map $f':S \times [0,\infty) \to X$.
\end{lemma}

\begin{proof}
The choice $f_{i}=f'_{i}$ will do, in case $S$ is finite. When $S$ is infinite, we can assume $S=\mathbb{Z}_{+}$. Then, choose in $X$ a sequence of compacta $\emptyset = K_{1} \Subset K_{2} \Subset K_{3} \Subset \cdots$ with $X=\cup_{j}K_{j}$. By properness of $f_{i}$, there exists $a_{i}$ in $[0,\infty)$ so large that $f_{i}( [a_{i},\infty) ) \subset X-K_{i}$. Define $f'_{i}$ to be $f_{i}$ precomposed with the retraction $[0,\infty) \to [a_{i},\infty)$.\ss

It is easily seen that $f'_{i}$ is properly homotopic to $f_{i}$. The properness of the resulting $f'_{i}$ now follows. Indeed, if $K \subset X$ is compact, then $K$ lies in the interior of $K_{i}$ for some $i$, hence $f'_{j}( [0,\infty) ) \cap K = \emptyset$ for $j > i$. Thus, the preimage $f'^{-1}(K)$ in $S \times [0,\infty)$ meets $j \times [0,\infty)$ only for $j\le i$. But, the intersection $f'^{-1}(K) \cap \{1,2,\dots,i\} \times [0,\infty)$ is compact by the finite case.
\end{proof}

Here is a key lemma concerning just one singular ray that will help to deal with infinitely many rays.

\begin{lemma}\label{l:onesingularray}
Let $K$ be a given compact subset of a locally compact, sigma compact space $X$ and let $f$ and $f'$ be singular rays in $X$ whose images are disjoint from $K$. If $f$ and $f'$ are properly homotopic in $X$, then the proper homotopy can be (re)chosen to have image disjoint from $K$.
\end{lemma}

\begin{proof}
If $h_{t}: [0,\infty) \to X$, $0 \leq t \leq 1$, is a proper homotopy from $f=h_{0}$ to $f'=h_{1}$, then its properness assures that for some $d \geq 0$ the image $h_{t}( [d,\infty) )$ is disjoint from $K$ for all $t$. But, the singular ray $f$ is proper homotopic in the complement of $K$ to the singular ray $\widehat{f}$ that is $f|_{[0,\infty)}$ precomposed with the retraction $[0,\infty) \to [d,\infty)$. Similarly for $f'$. Shunting together these three proper homotopies, one obtains the asserted proper homotopy.
\end{proof}

\begin{lemma}\label{l:manysingularrays}
Let $f_{i}$ and $f'_{i}$, with $i$ varying in the finite or countably infinite discrete set $S$, be two indexed sets of singular rays in the connected, locally compact, sigma compact space $X$. Suppose that the two continuous maps $f$ and $f'$ from $S \times [0,\infty)$ to $X$ defined by the rules $(i,x) \mapsto f_{i}(x)$ and $(i,x) \mapsto f_{i}'(x)$ are both proper. Suppose also that $f_{i}$ is proper homotopic to $f'_{i}$ for all $i \in S$. Then, there exists a proper homotopy $h_{t}:S \times [0,\infty) \to X$, $0 \leq t \leq 1$, that deforms $h_{0}=f$ to $h_{1}=f'$.
\end{lemma}

\begin{proof}
We propose to define the wanted proper homotopy $h_{t}$ by choosing, for $i \in S$, \emph{suitable} proper homotopies $h_{i,t}$ from $f_{i}$ to $f'_{i}$ and then defining $h_{t}$ by setting $h_{t}(i,x) = h_{i,t}(x)$ for all $i \in S$, all $t \in [0,1]$, and all $x \in [0,\infty)$. The choices aim to assure that $h_{t}$ is a \emph{proper} homotopy -- which means that the rule $(t,i,x) \mapsto h_{t}(i,x)$ is proper as a map $[0,1] \times S \times [0,\infty) \to X$.\ss

If $S$ is finite, any choices will do. But, if $S$ is infinite, then bad choices abound. For example, $h_{t}$ is not proper if every homotopy $h_{i,t}(x)$ meets a certain compactum $K$.\ss

If $S$ is infinite, we now specify choices that do the trick. Without loss, assume $S=\mathbb{Z}_{+}$. Let $\emptyset = K_{1} \Subset K_{2} \Subset K_{3} \Subset \cdots$ be an infinite sequence of compacta with $X=\cup_{j}K_{j}$. For each $i \in S$, let $J(i)$ be the greatest positive integer such that the images of the singular rays $f_{i}$ and $f'_{i}$ are both disjoint from $K_{J(i)}$. Since $f$ and $f'$ are proper, $J(i)$ tends to infinity as $i$ tends to infinity. Use Lemma \ref{l:onesingularray} to choose the proper homotopy $h_{i,t}$ from $f_{i}$ to $f'_{i}$ to have image disjoint from $K_{J(i)}$. Then, the properness of the resulting $h_{t}$ is verified as in the proof of Lemma \ref{l:rayproper}.
\end{proof}

\begin{remark}
Lemmas~\ref{l:rayproper} to~\ref{l:manysingularrays} above hold good with $[0,\infty)$ replaced by its product with (varying) compacta.
\end{remark}

Define a \dfb{multiray\/} in the \textsc{cat} manifold $M^{m}$ to be a \textsc{cat} submanifold lying in $\text{Int}M$ each component of which is a ray. Combining the Stable Range Embedding Theorem (Theorem~\ref{t:stableembedding}) with Lemmas~\ref{l:rayproper} to~\ref{l:manysingularrays} concerning proper maps, we get

\begin{proposition}[Classifying Multirays via Proper Homotopy]\label{p:raysviahomotopy}
\hfill\break
Let $M^{m}$ be a connected noncompact \textsc{cat} manifold, and let $f_{i}$ be singular rays where $i$ ranges over a finite or countably infinite index set $S$. If $m\geq3$, then $f_{i}$ is properly homotopic to a \textsc{cat} embedding $g_{i}$ onto a ray, such that the rules $(i,x) \mapsto g_{i}(x)$ collectively define a (proper) \textsc{cat} embedding $g:S \times [0,\infty) \to M$ with image a multiray. Furthermore, if $m>3$ and $g'_{i}$ is an alternative choice of the ray embeddings $g_{i}$, resulting in the alternative \textsc{cat} embedding $g'$ onto a multiray, then there exists an ambient isotopy $h_{t}:M \to M$, $0 \leq t \leq 1$, such that $h_{0}=\textnormal{id}|_{M}$ and $h_{1}g=g'$. \qed
\end{proposition}

\section{Multiple Component Hyperplane Embeddings}\label{s:multiplehyperplanes}

In this section we investigate proper \textsc{cat} embeddings into $\mathbb{R}^{m}$ of a disjoint sum of at most countably many\footnote{Every closed subset of a separable metric space is separable.} disjoint hyperplanes, each isomorphic to $\mathbb{R}^{m-1}$.\ss

For \textsc{cat}=\textsc{top} we will, for the first time, make essential use of the Stable Homeomorphism Theorem (=~SHT) to show that every self homeomorphism of $\mathbb{R}^{k}$ is ambient isotopic to a linear one~\cite{kirby69,freedmanquinn}; this is equivalent to $\pi_{0}( \text{STop} (k) ) = 0$, where STop$(k)$ is the group of orientation preserving self homeomorphisms of $\mathbb{R}^{k}$ endowed with the compact open topology. Not to do so would lead to pointless hairsplitting.\ss

In these circumstances, we can and do revert to the unrefined versions of the definition for \textsc{top} of the \textnormal{CSI} operation and its related constructions. We use the following lemma.

\begin{lemma}
If $G$ and $G'$ are \textsc{top} gaskets and $f: {\partial}G \to {\partial}G'$ is a degree $+1$ \textsc{top} isomorphism of their boundaries, then $f$ extends to a degree $+1$ \textsc{top} isomorphism $F:G \to G'$.
\end{lemma}

\begin{noqedproof}
By definition of gasket (see Section~\ref{s:csi}), we may assume $G$ and $G'$ are linear gaskets. By the SHT, we can isotop $f$ to a \textsc{diff} isomorphism $f'$. This $f'$ extends to a degree $+1$ \textsc{diff} isomorphism \hbox{$F':G \to G'$} by the \textsc{diff} version of this lemma (Corollary~\ref{c:gasketboundaryextension} above). Using closed collars of ${\partial}G$ and ${\partial}G'$ we easily construct the asserted \textsc{top} isomorphism $G \to G'$. \qed
\end{noqedproof}

A \dfb{multiple hyperplane} is a properly embedded submanifold $N$ of $\mathbb{R}^{m}$ where $N$ is the disjoint union of components $N_{i} \cong \mathbb{R}^{m-1}$ for $i \in S$, and $S$ is a nonempty countable index set. We say that $G$, the closure of a component of $\mathbb{R}^{m} - N$ in $\mathbb{R}^{m}$, is \dfb{docile} if it is a gasket, and we say that $N$ itself is \dfb{docile} if the closure of every such component is docile.\ss

Given any multiple hyperplane $N$ in $\mathbb{R}^{m}$, we can construct a canonical simplicial tree $T$ as follows. The vertices $V$ of $T$ are the closures of the complementary components of $N$ in $\mathbb{R}^{m}$. An edge is a component $N_{i}$ of $N$ and it joins the two vertices $u,v \in V$ whose intersection is $N_{i}$. The tree $T$ is clearly well defined by the pair $(\mathbb{R}^{m},N)$ up to tree isomorphism; it is the nerve of the covering of $\mathbb{R}^{m}$ by the closures of the components of $\mathbb{R}^{m} - N$. Also, $T$ is at most countable, but it is not necessarily locally finite. If $m = 2$, then these trees are naturally planar as the edges at each vertex are cyclically ordered.\ss

Conversely, given such a tree $T$ (planar in case $m=2$), there is a natural recipe to construct a multiple hyperplane $N$ in $\mathbb{R}^{m}$ where the closure of each complementary component is a gasket as follows. For each vertex $v_{k} \in V$, pick a gasket $G_{k}$ with boundaries corresponding bijectively to the edges incident with $v_{k}$ in $T$. Gluing these gaskets together according to $T$ gives a composite gasket $TG$ with empty boundary.\ss

It was established in proving the associativity property of \textnormal{CSI} that there is a \textsc{cat} manifold isomorphism $TG \to \mathbb{H}^{m}$ sending each vertex gasket in $TG$ to a linear gasket in $\mathbb{H}^{m}$ and, hence, each edge hyperplane to a hyperbolic hyperplane in $\mathbb{H}^{m}$ (see Lemma~\ref{l:TGisgasket} above). Further, such an isomorphism is unique up to degree $+1$ \textsc{cat} isomorphism of $\mathbb{H}^{m}$.\ss

We now summarize these observations, where \textsc{cat} is \textsc{top}, \textsc{pl} or \textsc{diff}.

\begin{theorem}[Multiple Hyperplane Linearization Thm. (=~MHLT)]\label{docilethm}
For $m$ distinct from $3$, every \textsc{cat} multiple hyperplane embedding $N$ in $\mathbb{R}^m$ is docile. Hence, for $m>3$, such embeddings are naturally classified modulo ambient degree $+1$ \textsc{cat} automorphism by arbitrary countable simplicial trees modulo simplicial tree automorphisms. For $m=2$ (and only $m=2$) one must use planar trees and their planar tree automorphisms (where planar here means that, at each vertex, the edges are cyclicly ordered).
\end{theorem}

\begin{corollary}[Gasket Recognition Theorem (=~GRT)]\label{grt}
\hfill\break
Consider a \textsc{cat} $m$-manifold with nonempty boundary whose interior is isomorphic to $\R^m$, and for which every boundary component is isomorphic to $\R^{m-1}$.
Exclude the case $m=3$. Then $M$ is isomorphic to a linear gasket. \qed
\end{corollary}

\begin{pfgrt}
$\Int \, M$ is always isomorphic to the manifold obtained by adding to $M$ an external open collar along $\partial{M}$. \qed
\end{pfgrt}

\begin{corollary}
With the same data as in the MHLT and assuming $m\geq4$, the pair $(\mathbb{R}^m,N)$ is \textsc{cat} isomorphic to a Cartesian product
\[
	(\mathbb{H}^2,N') \times \mathbb{R}^{m-2}
\]
where each component of $N'$ is a hyperbolic line. \qed
\end{corollary}

\begin{pfofdocilethmmgt3pldiff}
Let $G$ be the closure of a component of $\mathbb{R}^{m}-N$ in $\mathbb{R}^{m}$. Reindex so that $N_{i}$, $i\in S$, are the boundary components of $G$. For each $N_{i}$, let $V_{i}$ denote the closed component of $\mathbb{R}^{m}-\text{Int}G$ with boundary $N_{i}$. Each $N_{i}$ is unknotted in $\mathbb{R}^{m}$ by the \textsc{cat} HLT (Theorem~\ref{maintheorem}). Therefore, for each $i\in S$ there is a \textsc{cat} proper ray $r_{i}\subset \text{Int} V_{i}$ so that $V_{i}$ is a \textsc{cat} regular neighborhood of $r_{i}$ in $\mathbb{R}^{m}$. As $N\subset\mathbb{R}^{m}$ is a proper submanifold, the union of the rays $r_{i}$ is a proper multiray in $\mathbb{R}^{m}$.\ss

Choose $G'\subset\mathbb{H}^{m}$ a linear gasket with boundary hyperplanes $N'_{i}$, $i\in S$. For each $N'_{i}$, let $V'_{i}$ denote the closed component of $\mathbb{H}^{m}-\text{Int}G'$ with boundary $N'_{i}$ and let $r'_{i}\subset\text{Int}V'_{i}$ be a radial ray. Plainly, $V'_{i}$ is a \textsc{cat} regular neighborhood of $r'_{i}$ for each $i\in S$ and the union of the rays $r'_{i}$ is a proper multiray in $\mathbb{H}^{m}$.\ss

Choose a \textsc{cat} isomorphism $\psi:\mathbb{R}^{m} \rightarrow \mathbb{H}^{m}$. \textsc{cat} proper multirays unknot in $\mathbb{H}^{m}$, $m>3$, by the basic \textsc{cat} Stable Range Embedding Theorem (Theorem~\ref{t:stableembedding}), proved by general position, and Lemmas~\ref{l:simplestproperray} and~\ref{l:manysingularrays}. Thus, there is an ambient isotopy of $\mathbb{H}^{m}$ carrying $\psi(r_{i})$ to $r'_{i}$ for all $i\in S$ simultaneously. So, we may as well assume $\psi(r_{i})=r'_{i}$ for $i\in S$. By \textsc{pl} and \textsc{diff} regular neighborhood ambient uniqueness (see Section~\ref{s:regnhbds}), we may further assume that $\psi(V_{i})=V'_{i}$ for all $i \in S$. Then, $\psi|_{G}:G \rightarrow G'$ is a \textsc{cat} isomorphism as desired. \qed
\end{pfofdocilethmmgt3pldiff}

\begin{pfofdocilethmmgt3top}
Again, let $G$ be the closure of a component of $\mathbb{R}^{m}-N$ in $\mathbb{R}^{m}$ and reindex so that $N_{i}$, $i\in S$, are the boundary components of $G$. We have three cases depending on the number $\left|S\right|$ of boundary components of $G$.

\begin{case1comp}
This is exactly Cantrell's \textsc{top} HLT (Theorem~\ref{maintheorem}). \qed
\end{case1comp}

\begin{case2comp}
This case is well known as the \textbf{Slab Theorem} and is a
worthy sequel by C.~Greathouse~\cite{greathouse2} 1964 to
Cantrell's \textsc{top} HLT (Theorem~\ref{maintheorem}), so we include a proof. Greathouse deduced it from results then recently
established, together with the following (for $m>3$) then unproved.

\begin{theorem}[Annulus Conjecture (=~AC$(m)$)]\label{annconj}
\hfill\break
If $S_1$ and $S_2$ are two disjoint locally flatly embedded
$(m-1)$-spheres in $S^m$ and $X$ is the closure of the component of
$S^m - (S_1 \cup S_2)$ with $\partial{X}=S_1 \sqcup S_2$, then
$X$ is homeomorphic to the standard annulus $S^{m-1}\times[0,1]$.
\end{theorem}

\noindent This Annulus Conjecture was later proved, along with the
SHT, in \cite{kirby69} 1969 for $m>4$, and in~\cite{freedmanquinn} 1990 for $m=4$ (see also~\cite{edwards}). The already proved results used in~\cite{greathouse2} included Cantrell's \textsc{top} HLT, that we have reproved (Theorem~\ref{maintheorem}), and the following, proved by Cantrell and Edwards~\cite{cantrelledwards} 1963.

\begin{lemma}[Arc Flattening Lemma]\label{arcflatlemma}
\hfill\break
If a compact arc $A$ topologically embedded in $S^m$, $m>3$,
is locally flat except possibly at one interior point
$P$, then $A$ is locally flat also at $P$.
\end{lemma}

Assuming these tools for the moment, we now give:

\begin{slabproof}
We consider the sphere $S^m$
to be $\mathbb{R}^m \cup \infty$. Let $G_i$, $i=1,2$, be the components of $\mathbb{R}^{m} - \Int G$. By the \textsc{top} HLT (Theorem~\ref{maintheorem}), each $G_i \cong \mathbb{R}_{+}^{m}$. Hats will indicate the adjunction of the
point $\infty \in S^m$. Enlarge $\widehat{G}:=G\cup\infty$
by adding to it a closed collar $C_i$ of the
$(m-1)$-sphere $\partial{\widehat{G}_i}$ in $\widehat{G}_i$, for $i=1,2$.
Denote the result $X:= \widehat{G} \cup C_1\cup C_2$. This is a
\textsc{top} submanifold of $S^m$ with boundary two
$(m-1)$-spheres $S_1$ and $S_2$, where $S_i$, for  $i=1$ and $2$,  is the component of $\partial{C_i}$ disjoint from $\widehat{G}$ (see Figure~\ref{slabfig}).
\begin{figure}[h!]
	\centerline{\includegraphics{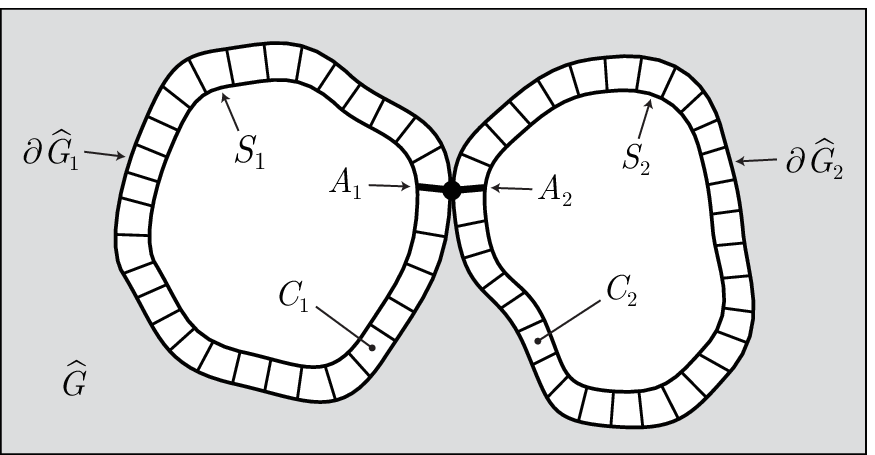}}
	\captionsetup{labelformat=ss,textfont={it,small},labelsep=newline,justification=centerlast,width=\textwidth}
	\caption{Almost global view of $\widehat{G}$ in $S^m = \mathbb{R}^m \cup \infty$, focused on the point
	$\infty = A_1 \cap A_2$.}
	\label{slabfig}
\end{figure}

The theorem AC$(m)$ (Theorem~\ref{annconj}) tells us that
$X \cong S^{m-1}\times[0,1]$. Furthermore, we have collaring
identifications $C_i = S_i\times[0,1]$. Consider the locally
flat arc $A_i$ that is the arc fiber of the collaring $C_i=S_i\times[0,1]$
that contains $\infty \in S^m$. Clearly
$A_1\cap A_2 = \infty$; thus $A := A_1\cup A_2$ is an arc
in $X$ that is locally flat except possibly at $\infty \in S^m$. By
the above Arc Flattening Lemma (Lemma~\ref{arcflatlemma}), $A$ is locally flat at
$\infty$; hence it is a locally flat $1$-submanifold of $X$
joining the two boundary components of $X$.  Note that
$G\cong X - A$ by Brown's collaring uniqueness theorem~\cite{brown62}.\ss

By the uniqueness clause of the elementary (but subtle!)
\textsc{top} version of the Stable Range Embedding Theorem (Theorem~\ref{t:stableembedding}), any two such arcs are
related by a \textsc{top} automorphism of $X\cong S^{m-1}\times[0,1]$.
Thus the complement $X - A$ is homeomorphic to $\mathbb{R}^{m-1}\times[0,1]$. \qed
\end{slabproof}

\begin{proofarc}
Split $A$ at $P$ to get
two compact arcs $A_1$ and $A_2$ with $A_1 \cap A_2 = P$.

\begin{assertion}\label{goodballnhbd}
There exists a compact locally flat $n$-ball
neighborhood $B$ of $\Int \, A_1$ such that $A_1$ is unknotted in $B$
and $B$ is disjoint from $\Int \, A_2$.
\end{assertion}

\begin{onpfofass}
In our one application of the Arc Flattening Lemma
above (namely to prove the Slab Theorem), $B$ can obviously be any tubular neighborhood
of $A_1$ in $C_1$ derived from the product structure 
$C_1 = S_i \times [0,1]$. Thus we leave
the full proof of this assertion to the interested reader with
just this hint: $B$ can in general be the closure in $S^{m}$
of a suitably tapered trivial normal tubular neighborhood
of $\Int A_1$ in $S^m$ (see~\cite{lacher67}). \qed
\end{onpfofass}

Now, by the \textsc{top}
Schoenflies theorem, $S^{m} - \Int B$ is also an $m$-ball $B'$
in $S^m$. In $B'$ the second arc $A_2$ is embedded in a
manner that is locally flat except possibly at $P\in\partial{B'}$. To
the non-compact \textsc{top} manifold $B' - P\cong\mathbb{R}_{+}^{m}$ we apply the
uniqueness clause of the Stable Range Embedding Theorem (Theorem~\ref{t:stableembedding});
we conclude, on recompactifying in $S^m$,
that the arc $A_2$ is unknotted in $B'$. It follows 
that the arc $A := A_1 \cup A_2$ is locally flat in
$S^m$. This completes our proof of the Arc Flattening Lemma (Lemma~\ref{arcflatlemma}). \qed
\end{proofarc}

Assuming AC$(m)$ (now known!), this completes the proof of the Slab Theorem which is the MHLT 
(Theorem~\ref{docilethm}) for the case when $G$ has two boundary
components. \qed
\end{case2comp}

\begin{remarks}
\Item(1)! Greathouse~\cite{greathouse1} 1964 also proved that the Slab
Theorem in dimension $m$ implies AC$(m)$, granting results
known in 1964 that we have mentioned. Hints: given 
an $m$-annulus $X$ in $S^m$,
form a locally flat arc $A \subset X$ joining the two
boundary \hbox{$(m-1)$-spheres}. Show that $A$ is \emph{cellular} (i.e., an intersection of compact $m$-cell neighborhoods in $S^m$) so
that the quotient space $(S^m /A)$ is homeomorphic to $S^m$, and apply the Slab Theorem to show
that $X - A\cong \mathbb{R}^{m-1}\times[0,1]$. Deduce that 
$X \cong S^{m-1}\times[0,1]$ with the help of collarings and
the Mazur-Brown Schoenflies theorem.
\medskip

\Item(2)! An easy argument shows that AC$(n)$ (Theorem~\ref{annconj}), $n=1,\ldots,m$, together 
imply the following.

\begin{theorem}[Stable Homeomorphism Conjecture (=~SHC$(m)$)]\label{shc}
\hfill\break
For any homeomorphism $h:\mathbb{R}^m \to \mathbb{R}^m$, there exists a 
homeomorphism $h':\mathbb{R}^m \to \mathbb{R}^m$ that coincides with $h$ near the 
origin and with the identity map outside a bounded set.
\end{theorem}

\noindent Hint: For this implication, you will need some 
Alexander isotopies. Exactly this form of the SHC$(m)$ was proved for $m \geq 5$ by R.~Kirby in~\cite{kirby69}.
\medskip

\Item(3)! An easy argument establishes the implication SHC$(m)$ $\Rightarrow$ AC$(m)$.
\end{remarks}
\bigskip

\begin{casegt2comp}
Let $G$ be the closure in $\mathbb{R}^m$ of a component of $\mathbb{R}^{m}-N$. Let $N_i$, $i\in S$, be an indexing of the components of $\partial{G}$. For each $N_{i}$, let $V_{i}$ denote the closed component of $\mathbb{R}^{m}-\text{Int}G$ with boundary $N_{i}$. By the \textsc{top} HLT (Theorem~\ref{maintheorem}), each $V_{i}$ is \textsc{top} isomorphic to closed upper half space $\mathbb{R}^{m}_{+}$. It is straightforward to produce, for each $i\in S$, a \textsc{diff} proper ray $r_{i}\subset\text{Int}V_{i}$. Let $T(r_{i})\subset\text{Int}V_{i}$ be a \textsc{diff} (closed) regular neighborhood of $r_{i}$. The boundary $H_{i}$ of $T(r_{i})$ is a \textsc{diff} hyperplane. By the Slab Theorem, the closure of the region between $H_{i}$ and $N_{i}$ is \textsc{top} isomorphic to $\mathbb{R}^{m-1}\times [0,1]$. This isomorphism yields an obvious isotopy of $N_{i}$ to $H_{i}$ for each $i\in S$. Using disjoint collars of the $H_{i}$ and $N_{i}$, these isotopies readily extend to an ambient isotopy of $\mathbb{R}^{m}$ which carries the collection $N_i$, $i\in S$, to the \textsc{diff} collection $H_i$, $i\in S$. The result now follows from the \textsc{diff} proof of the MHLT (Theorem~\ref{docilethm}) above. \qed
\end{casegt2comp}

This completes the proof of the MHLT (Theorem~\ref{docilethm}) for $m>3$ and \textsc{cat}=\textsc{top}. \qed
\end{pfofdocilethmmgt3top}

\begin{pfofdocilethmmeq2}
We begin with the following.\ss

\begin{observation}
By triangulation and smoothing
theorems for dimension $2$ that we refer collectively as
the $2$-Hauptvermutung (see \cite{moise} or~\cite{siebenmann05}),
it suffices to establish the MHLT (Theorem~\ref{docilethm}) for any one of the three
categories \textsc{cat} = \textsc{diff}, \textsc{pl}, or \textsc{top}.
\end{observation}

We work in the smooth category. The \textsc{diff} proof of the MHLT (Theorem~\ref{docilethm}) already given for $m>3$ adapts easily to $m=2$ using the following.

\begin{theorem}[Multiray Radialization Theorem in $\mathbb{R}^2$ (=~$2$-MRT)]\label{mrt}
Let $L\subset\mathbb{R}^2$ be a smooth proper multiray. Then $L$ is ambient isotopic to a radial multiray.
\end{theorem}

\begin{pfofmrt}
Translate so that $L$ misses the origin. Morse theory tells us that, by a small smooth perturbation of $L$ in $\mathbb{R}^2$,
we may assume that
\[
\xymatrix@R=0pt{
	\mathbb{R}^2	\ar[r]^-{f}		&		\mathbb{R}\\
	x				\ar@{|-{>}}[r]	&		\left|x\right|}
\]
restricts to a Morse function on $L$ with distinct critical values, cf.~\cite{milnor65}.\ss

Seen in a nutshell, our proof plan is to alter $L$ by ambient isotopy to eliminate critical points of $f|_L$ in pairs until $f|_L$ has no critical points at all. Then a further ambient isotopy will make $L$ radial; this latter isotopy will arise by integration of a vector field on $\mathbb{R}^2$ that is tangent to $L$ (as modified so far) and is transverse to the level spheres of $f$.

\begin{assertion}\label{henrya1}
By an ambient isotopy, we may assume that on each component $r$ of $L$ an absolute minimum of the restriction $f|_r$ is attained at the point $\partial{r}$ only, and this point is
noncritical for $f|_r$.
\end{assertion}

\begin{pfhenrya1}
For each component $r_i$ of $L$ for which $f(\partial{r_i})$ is not
the unique minimum point $m_i$ of $f$ on $r_i$, consider a small smooth
regular neighborhood $N_i$ of the interval $K_i$ in $r_i$ that joins $m_i$
to $\partial{r_i}$. These $N_i$ can be chosen so small that their union $N$ is
a disjoint sum of these $N_i$. Then independent smooth isotopies, each
with support in one $N_i$, together establish the assertion (cf.~\cite[pp.~22--24]{milnor97}). \qed
\end{pfhenrya1}

Now, perform a possibly infinite number of steps. For convenience, we consider the points in $\partial{L}$ to be critical points of $f|_L$ from here on.\ss

\begin{step1}
Let $u_{\mathrm{max}}$ be the local maximum of $f|_L$ on which $f$ assumes the minimum value. Let $r_i$ be the component of $L$ containing $u_{\mathrm{max}}$.
Let $u_0$ and $u_1$ be the critical points of $f|_L$
adjacent to $u_{\mathrm{max}}$ in $r_i$; the only critical points in the segment from $u_0$ to $u_1$ in $r_i$
are $u_0$, $u_{\mathrm{max}}$, and $u_1$.  
After switching $u_0$ and $u_1$ if needed, we may assume 
$f(u_1) > f(u_0)$.  
There is a unique
$u_2$ on the segment of $r_i$ from $u_0$ to $u_{\mathrm{max}}$ so that 
$f(u_2) = f(u_1)$.\ss

We will ambiently cancel the index $0$ and index $1$ critical points $u_1$ and $u_{\mathrm{max}}$.
Let $A$ denote the complement in $\mathbb{R}^2$ of the open disk of radius $f(u_1)$ centered at the origin.
Let $D$ be the compact region in $A$ bounded by the segment
of $r_i$ from $u_1$ to $u_2$ and an arc of the circle of radius
$f(u_1)$ between $u_1$ and $u_2$ (see Figure~\ref{cancel}).
\begin{figure}[h!]
	\centerline{\includegraphics{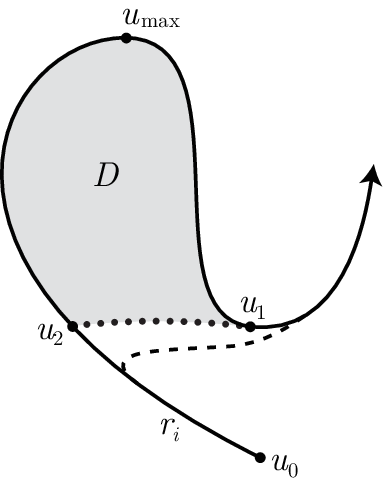}}
	\captionsetup{labelformat=ss,textfont={it,small},labelsep=newline,justification=centerlast,width=.88\textwidth}
	\caption{Cancelling pair of critical points $u_1$ and $u_{\mathrm{max}}$. The dotted arc indicates an arc of the circle of radius $f(u_1)$ between $u_1$ and $u_2$. The dashed line indicates the trajectory of $h_1(r_i)$ after cancellation.}
	\label{cancel}
\end{figure}

\begin{assertion}\label{henrya2}
For any neighborhood $U$ of $D$ there is a diffeomorphism $h_1$ of $\mathbb{R}^2$ so that:

\Item(a)! $f(h_1(x)) \le f(x)$ for all $x \in \mathbb{R}^2$.

\Item(b)! The support of $h_1$ lies in $U$.

\Item(c)! The support of $h_1$ does not intersect $L - r_i$ and also only intersects $r_i$ in a 
		small neighborhood of the segment of $r_i$ from $u_1$ to $u_2$.
		
\Item(d)! The critical points of the restriction of $f$ to $h_1(L)$ are the same as those of $f|_L$,
		except for $u_{\mathrm{max}}$ and $u_1$ which are no longer critical or even in $h_1(L)$.
\end{assertion}

\begin{pfhenrya2}
Note that the interior of $D$ does not intersect $L$,
since that would give a local maximum of $f|_L$ 
with value $< f(u_{\mathrm{max}})$, contrary to our choice of $u_{\mathrm{max}}$.
Consequently, we may assume $U$ does not intersect  $L-r_i$
and does not
intersect $r_i$ outside a small neighborhood of the segment of $r_i$ from $u_1$ to $u_2$.\ss

We get $h_1$ by integrating a suitable vector field $v$ on $\mathbb{R}^2$.
All we need to do is make sure that:
\smallskip

\noindent$\bullet$ $v(x)\cdot x\le 0$ for all $x \in \mathbb{R}^2$  (to get (a)).
 
\noindent$\bullet$ $v(x)\cdot x = -|x|$ for all $x$ in some neighborhood $U'$ of $D$.
 
\noindent$\bullet$ $v$ points into $D$ on the interior of the segment of $r_i$ from $u_1$ to $u_2$.

\noindent$\bullet$ $v$ points out of $D$ on the arc of the circle of radius $f(u_1)$ from $u_1$ to $u_2$.

\noindent$\bullet$ The support of $v$ is contained in $U$, does not
intersect $L-r_i$, and does not intersect $r_i$ outside of a 
small neighborhood of the segment of $r_i$ from $u_1$ to $u_2$.
\smallskip

It suffices to find such a $v$ locally, since then $v$ is obtained by piecing together with a partition of unity. Finding $v$ locally is easy. The vector field $v(x)=-x/|x|$ works on the interior of $D$, on the circle of radius $f(u_1)$ (except possibly at $u_2$), 
and near $u_{\mathrm{max}}$.
On the segment of $r_i$ from $u_1$ to $u_2$ (but not at $u_{\mathrm{max}}$ or $u_1$), one may take a 
tangent to $r_i$ plus a small inward pointing vector.\ss

Having obtained a vector field $v$ satisfying the above
conditions, we may construct $h_1$ by standard methods.
More precisely, suppose $g: [a,b]\to U'$ parameterizes
a slightly larger  segment of $r_i$ than the segment from
$u_2$ to $u_1$. Choose a smooth function $\alpha:[a,b] \to [0,\infty)$
with support in $(a,b)$ so that $fg-\alpha$ is very small
and has positive derivative everywhere; this is possible since
$f(a)$ is slightly less than $f(u_2)$ and $f(b)$ is slightly greater than
$f(u_1)=f(u_2)$. Let $\phi(x,t)$ be the flow associated to the
vector field $v$ and let $\gamma:\mathbb{R}\to \mathbb{R}$ be a smooth
function with compact support such that $\gamma(0)=1$.
Then we define $h_1$ to be the identity outside
$\phi(g((a,b))\times \mathbb{R})$ and we define
$h_1\phi(g(s),t) = \phi(g(s),t +  \alpha(s)\gamma(ct))$ for some appropriate $c>0$.  If \hbox{$c>0$} is chosen small enough, then
the mapping $(s,t)\mapsto (s,t+\alpha(s)\gamma(ct))$ is a diffeomorphism, so $h_1$ is a diffeomorphism.
Note that $h_1(L)$ is obtained from $L$ by replacing the
segment $g([a,b])$ with the segment $\{ \phi(g(s),\alpha(s)) \mid s\in [a,b]\}$.
Since $f\phi(g(s),\alpha(s)) = fg(s)-\alpha(s)$ which has nonzero 
derivative, the replacement
segment has no critical points of $f$. This completes the proof of Assertion~\ref{henrya2}. \qed
\end{pfhenrya2}

\end{step1}

Step~$1$ is complete. \qed

\begin{stepk}
Similar to Step~$1$, but we replace $L$ by $h_{k-1}h_{k-2}\cdots h_1(L)$ and we name the resulting diffeomorphism $h_k$.
To ensure well definition of the infinite composition (discussed directly below), we insist that:\ss

\smallskip

\Item(\ddag)! The support of $h_k$ lies outside the disk of radius $f(u_1)-1$.
\end{stepk}

Step~$k$ is complete. \qed

\medskip

Note that the infinite composition 
$h = \cdots h_k h_{k-1}\cdots h_2 h_1$ is a diffeomorphism of $\mathbb{R}^2$.
To see this, pick any $r>0$.
Using (\ddag), one can show that there is an $n$ so that the support of $h_k$ does not
intersect the disc of radius $r$ if $k>n$.
Then $h(x) = h_n h_{n-1}\cdots h_1(x)$ for all $x$
with $|x|<r$ since $|h_n h_{n-1}\cdots h_1(x)|\le |x| < r$.\ss

After replacing $L$ with $h(L)$ we may as well assume
that the restriction of $f$ to $L$ has no critical points except the 
obligatory $\partial{L}$.
It is now easy to straighten $L$ using a slightly modified
proof of the isotopy extension theorem.
After a homothety we may assume
   $L$ does not intersect the
disc of radius $1$ about the origin.
Let $v$ be a
  vector field on $\mathbb{R}^2-0$ tangent to $L$ and satisfying $v(x)\cdot 
\nabla f = 1$
for all $x$, and $v(x) = \nabla f  = x/|x|$ for $|x|<1$.
If $\phi(x,t)$ is the flow for this vector field,
then $f\phi(x,t) = f(x)+t$ since $d f\phi(x,t)/dt = v\cdot \nabla f = 
1$.
Consequently, any trajectory beginning from a
point $x \in \mathbb{R}^2 - 0$ at time $t = 0$ crosses
the radius $r$ circle at exactly
time $t = r - {\vert x \vert}$.
Tangency to $L$ guarantees that if $\phi(x,t)\in L$ then $\phi(x,s)\in 
L$ for all $s\ge t$, and in fact for all $s\ge t_0$
where $\phi(x,t_0)\in \partial{L}$.
In particular, if $r_i$ is any ray of $L$ with boundary $x_i$,
then $r_i$ is exactly the positive trajectory of $x_i$.
Define $g\colon \mathbb{R}^2 \to \mathbb{R}^2$ by $g(0)=0$ and
$g(x) = |x| \phi(  x, 1 - |x|)$ for $x\ne 0$,
i.e., flow to the circle of radius one and then
scale back to your original distance from the origin.
Note that
$\phi(x,t) = x(1+t/|x|)$ for $|x|<1$ and $t\le 1-|x|$ so we know
$g(x)=x$ for $|x|<1$ and thus $g$ is smooth at the origin.
Tangency of $v$ to $L$ guarantees that for any ray $r_i$
of $L$, then $\phi(x, 1-|x|)$ is the same for all $x\in r_i$
and hence $g(r_i)$ is radial.\ss

The proof of the $2$-MRT (Theorem~\ref{mrt}) is complete. \qed
\end{pfofmrt}

This completes the proof of the MHLT (Theorem~\ref{docilethm}) for $m=2$. \qed
\end{pfofdocilethmmeq2}

\begin{remarks}\label{2mrtremarks}
\Item(1)! The basic technique used in the above proof of the $2$-MRT (Theorem~\ref{mrt}) is to ambiently cancel a minimal height local maximum of $f|_L$ with an adjacent local minimum in a controlled fashion. This technique is noteworthy for its simplicity and its utility. The \textsc{diff} Schoenflies theorem was nowhere employed as only basic separation properties are needed to obtain the vector field. Indeed, this technique quickly yields proofs of both the \textsc{diff} Schoenflies theorem and the \textsc{diff} HLT (Theorem~\ref{t:ourunknottingtheorem}) for $m=2$, the latter without assuming any ray hypothesis, as we now describe.

\begin{pf2dschoen}
Let $K$ be a smooth circle in the plane.
Perturb $K$ so that $f(x)=|x|$ is Morse on $K$ with distinct critical values.
Let $m$ and $M$ be the absolute minimum and maximum points of $f$ on $K$.
Now, apply the above technique to the two segments of $K$ connecting $m$ and $M$. \qed
\end{pf2dschoen}

\begin{pfhlt}
Let $N$ be a smooth proper embedding of $\mathbb{R}^1$ in $\mathbb{R}^2$.
Since any smoothly embedded $1$-disk unknots in $\mathbb{R}^2$, we can arrange that $N$ coincides with the $x$-axis in an $\varepsilon$-neighborhood $N_\varepsilon$ of the origin.
Now, apply the above technique to the two rays $N-\Int\,N_\varepsilon$. \qed
\end{pfhlt}

\Item(2)! It is natural to consider the $n$-dimensional analog of the $2$-MRT (Theorem~\ref{mrt}), namely:

\begin{theorem}[Multiray Radialization Theorem (=~$n$-MRT)]\label{nmrt}
\hfill\break
Let $L\subset\mathbb{R}^n$ be a smooth proper multiray. If $n\neq3$, then $L$ is ambient isotopic to a radial multiray.
\end{theorem}

\noindent Recall that the $n$-MRT is `false' in dimension $n=3$ because even one proper ray may knot in $\mathbb{R}^3$ (see Remarks~\ref{sec7remarks}). On the other hand, if $n>3$, then the $n$-MRT holds (any \textsc{cat}) by the argument in the third paragraph of the proof of the MHLT (Theorem~\ref{docilethm}) for $m>3$ and \textsc{cat}=\textsc{pl} or \textsc{cat}=\textsc{diff} given earlier in this section; for \textsc{cat}=\textsc{top}, this argument uses Homma's method.\ss

We mention that for \textsc{cat}=\textsc{diff} and $n>3$, one may prove the \hbox{$n$-MRT} via the basic technique used in the above proof of the $2$-MRT (Theorem~\ref{mrt}). Indeed, this approach works with $\mathbb{R}^n$ replaced by any smooth manifold $W$ that is collared at infinity.
By ray shortening one can assume without loss that $W = M\times [0,\infty)$.
We claim that $L$ may be straightened, i.e., there is an ambient isotopy of $W$ carrying each ray of $L$ to a ray
of the form $m\times [t,\infty)$.
Since $n > 3$, one can slightly perturb
$L$ so that its projection to $M$ is a one to one immersion. This
canonicly provides a Whitney $2$-disk $D$ for suppression of a pair $u_{\mathrm{max}}$
and $u_1$ of critical points, cf. Figure~\ref{cancel}; indeed $D$ is made up of vertical
segments (just two degenerate), and the vector field is vertical. One
then concludes as for Theorem~\ref{mrt}. We need not process the $u_{\mathrm{max}}$ in min max order but
we do need to ensure that $h_k$ does not increase the
$[0,\infty)$ coordinate, as this guarantees the infinite
composition $\cdots h_k\cdots h_1$ is a diffeomorphism. The interested reader may enjoy
seeing where this argument fails in ambient dimension $n = 3$; an infinite number of
trefoils tied in a ray reveals the problem (a single trefoil tied in a ray reveals the local problem). One cannot make the projection of $L$ to $M$
one to one and thus may no longer exclude $L$
from the interior of the Whitney disc.
\end{remarks}

\nss{Two Alternative Proofs of the MHLT for Dimension $2$}

We have seen in the proof of the MHLT (Theorem~\ref{docilethm}) in
this section that it suffices to give alternative proofs
that each closure $M^2$ in $\mathbb{R}^2$ of a complementary component
of a properly embedded family of lines in $\mathbb{R}^2$ is
isomorphic to a linear gasket.
Thus it suffices to give new proofs of the
Gasket Recognition Theorem~\ref{grt} for dimension $2$, that we
restate as

\begin{theorem}[$2$-Gasket Recognition Theorem (=~$2$-GRT)]\label{2grt}\hfill\break
Consider a \textsc{pl} $2$-manifold $M$ 
whose interior is isomorphic to $\mathbb{R}^2$, and of which every
boundary component is non-compact. 
Then $M$ is isomorphic to $\mathbb{R}^2$, or to a linear gasket.
\end{theorem}

Note that the converse of Theorem~\ref{2grt} is trivial.\ss

\medskip

We pause to offer a broader understanding of this result.
We accept as known the following analog for dimension $2$ of the Poincar\'e conjecture:

\begin{classicalfact}[=~$2$-PC]\label{2pc}
Every compact $2$-manifold $N^2$ having $\mathrm{H}_1(N;\,\mathbb{Z}/2\mathbb{Z})=0$ is isomorphic to the sphere $S^2$ or to the disk $\mathbb{B}^2$.
\end{classicalfact}

This $2$-PC is part of almost any classification of compact \textsc{pl} (or \textsc{diff})
surfaces; see for example Section~9 of~\cite{hirsch}. \qed\ss

Aiming to analyse the hypotheses of Theorem~\ref{2grt} (=~$2$-GRT), we prove:

\begin{proposition}\label{equivcond}
Consider a connected non-compact $2$-manifold $M^2$.
The following conditions are equivalent:\smallskip

\Item(a)! $\Int \, M \cong \mathbb{R}^2$.

\Item(b)! $M$ is irreducible; in other words every circle \textsc{pl} embedded in $M$ is the boundary of a \textsc{pl} $2$-disk embedded in $M$.

\Item(c)! $M$ is contractible.

\Item(d)! $\mathrm{H}_1(M;\,\mathbb{Z}/2\mathbb{Z})=0$.

\end{proposition}

\begin{pfofequivcond}
Note that all four conditions are 
invariant under deletion (or addition) of boundary. 
Thus, without loss we can and do assume for the proof
that $\partial{M} = \emptyset$, i.e. $M$ is `open'.\ss

We can and do choose to work in the \textsc{pl} category.\ss 

By the \textsc{pl} Schoenflies theorem, (a) implies (b).
Trivially, (a) implies (c). By (PLCT) in Section~7 of~\cite{siebenmann05}, (b) implies (a).
By the homotopy axiom for
homology, (c) implies (d). To conclude, we prove that (d) implies (b).\ss

Consider any circle $C$ that is \textsc{pl} embedded in $M$. This
$C$ is bicollared, for otherwise its regular neighborhood is
a M\"obius band, which shows that $C$ has self-intersection
number $1$, and hence the class of $C$ is non-zero in
$\mathrm{H}_1(M;\, \mathbb{Z}/2\mathbb{Z})=0$, a contradiction.\ss

Continuing the proof that (d) implies (b), we examine several cases.\ss

Case~1. $C$ does not separate $M$.\hfill\break
Then there exists another embedded curve $C'$ in $M$ that
intersects $C$ in a single point and transversally. Thus $C$ and
$C'$ have mod $2$ intersection number $1$ in $M$. This shows that
the homology classes of $C$ and $C'$ in $\mathrm{H}_1(M;\,\mathbb{Z}/2\mathbb{Z})$ are both
nonzero, which contradicts (d). Thus, Case~1 cannot occur.\ss

Complementary Case~2. $C$ separates $M$.\hfill\break
Then, as $C$ is bicollared, it necessarily cuts $M$ into two connected pieces $M_1$
and $M_2$ each with boundary a copy of $C$. We now treat two subcases of Case~2
separately.\ss

Subcase~(i). Neither piece $M_i$ is compact.\hfill\break
Seeking a contradiction, suppose this subcase occurs.
There then exists a properly embedded path $C'$ in $M$ that
intersects $C$ in a single point and transversally.
There is thus a nonzero mod 2 intersection number of
$C$ with $C'$ proving that the class of $C$ in $\mathrm{H}_1(M;\,\mathbb{Z}/2\mathbb{Z})$
is nonzero, a contradiction. Thus this subcase cannot occur. We conclude that the following must always occur.\ss

Complementary Subcase~(ii). One piece, say $M_1$, is compact.\hfill\break
Then we claim that $\mathrm{H}_{1}(M_1;\,\mathbb{Z}/2\mathbb{Z})=0$. To prove this claim, suppose the contrary. Capping $M_1$ with a $2$-disk $B$
yields a \textsc{pl} closed $2$-manifold $N_1$ with
\[
	\mathrm{H}_1(N_1;\,\mathbb{Z}/2\mathbb{Z}) \cong\mathrm{H}_1(M_1;\,\mathbb{Z}/2\mathbb{Z}) \neq 0.
\]
In $\mathrm{H}_1(N_1;\,\mathbb{Z}/2\mathbb{Z})$, Poincar\'e duality provides a pair of compact curves $C_1$ and $C'_1$ (disjoint from $B$ by
general position) having non-zero intersection
number mod $2$. They lie in both $M$ and $M_1$ and have the same non-zero intersection number in $M$ as in $M_1$, contradicting $\mathrm{H}_1(M;\,\mathbb{Z}/2\mathbb{Z})=0$. This proves the claim.\ss

Next, since \hbox{$\mathrm{H}_{1}(M_1;\,\mathbb{Z}/2\mathbb{Z})=0$}, the classical $2$-PC tells us that $M_1$ is a $2$-disk. This proves for $M$ the irreducibility condition (b), and thereby completes the proof of
Proposition~\ref{equivcond}. \qed
\end{pfofequivcond}

\medskip
By the above Proposition~\ref{equivcond}, the following assertion is equivalent to $2$-GRT.

\begin{assertion}\label{2ga}
Every noncompact contractible $2$-manifold $M^2$ is isomorphic to a linear gasket in $\mathbb{H}^2$, or to $\mathbb{H}^2$ itself.
\end{assertion}

To conclude we present two quite different proofs of this assertion.

\begin{sctp}
We can assume \textsc{cat} = \textsc{top}.
The case of the assertion where $M$ has $3$ boundary 
components readily implies it for $2$, $1$, or $0$ boundary 
components, so we assume $M$ has $\geq 3$ boundary components.\ss

We shall use a pleasant \textsc{top} classification of such $M^2$
stated below. It is an easy consequence of three difficult
classical theorems applied to the double $DM$ of $M$ formed from
two copies of $M$ with  their boundaries identified. 
More details and references are given in~\cite{siebenmann08}.\ss

The first theorem was discovered by A.~Schoenflies~\cite{schoenflies02}
and states that a compact connected subset $J$
of the plane is a circle if and only if its complement has
two components and each point of $J$ is accessible as the
unique limit of a path in each. The second is the 
Osgood-Schoenflies theorem (proved circa 1912, see~\cite{siebenmann05})
stating that every circle $J$ topologically
embedded in the plane bounds a topological $2$-disk. The
third is due to B.~K\'er\'ekjart\'o~\cite{kerekjarto} and
classifies \emph{all} surfaces \emph{without} boundary, in particular $DM$, in terms
of what is now known as the (K\'er\'ekjart\'o-Freudenthal) \emph{end
compactification}.\ss

\begin{classification}
The end compactification of a noncompact contractible
surface $M$, written 
$E(M) = M \cup e(M)$, is  always
a $2$-disk, whose interior is $\Int M$, and  whose
boundary circle $\partial{E(M)}$ is the disjoint union 
\hbox{$\partial{M} \cup e(M)$} where  $e(M)$ is the compact and
totally disconnected \emph{end space} of $M$. Thus $M$
is homeomorphic to a $2$-disk $E(M)$ \emph{minus} a 
compact part $e(M)$ of its boundary.
\end{classification}

\begin{poc}
By~\cite{kerekjarto}, the end compactification $E(DM)$ is $S^2$. Then~\cite{schoenflies02}
shows that the obvious involution $\tau$ on $E(DM)$ has fixed point set a Jordan curve, and finally the Osgood-Schoenflies Theorem shows that $S^2/\tau = E(M)$ is a $2$-disk as required. \qed
\end{poc}

The remainder of the proof of Assertion~\ref{2ga} is elementary.
Identify $E(M)$ to the round euclidean disk  
$\mathbb{B}^2 \subset \R^2$ and consider the convex hull 
$\Hull(e(M))$ in $\R^2$. Since $M$ has $\geq 3$ ends, the convex hull
$\Hull(e(M))$ is topologically a $2$-disk in $\R^2$, and all
its extremal points (as a convex subset of $\mathbb{R}^2$) constitute $e(M) \subset \partial{\mathbb{B}^2}$.
Hence there is a standard homeomorphism 
$\Hull(e(M)) \to \mathbb{B}^2$, respecting every ray emanating
from the barycenter of the hull, and fixing $e(M)$. Thus
$M$ itself is \textsc{top} isomorphic to the linear gasket
\[  
   \Hull(e(M)) \cap \Int \mathbb{B}^2 = \Hull(e(M)) \cap \mathbb{H}^2.
   \eqno \qed
\]
\end{sctp}

\medskip

\begin{sgp}
There is a famous procedure that tiles any closed $2$-manifold $M_g$ of genus $g \geq 2$ by compact hexagonal $2$-cells (= tiles), and then constructs a hyperbolic structure for $M_g$ in which each $2$-cell has geodesic edges and all vertex angles $\pi/2$. In reply to our enquiry about known geometric proofs, J.-P.~Otal promptly suggested that a similar approach would prove the assertion.\ss

The case of the assertion for $\geq 3$ boundary components implies the general 
case, so \emph{we restrict to this case in what follows}.\ss

We work in the \textsc{diff} category.\ss

Given an arbitrary enumeration of the components of
$\partial{M}$ (called \emph{sides} below), there is a construction procedure of `cut and paste' topology to
construct on $M$ a \textsc{diff} tiling in which each $2$-dimensional
tile is closed and is either a compact hexagonal tile or a noncompact
cusp tile (= a triangle with one ideal vertex at
Alexandroff's infinity).\ss

These tiles will fit together as follows. Each finite vertex lies in $\partial{M}$. Each hexagonal tile $H$
has $3$ of its $6$ edges alternatively in three distinct
sides of $\partial{M}$ and the remainder of $\partial{H}$ lies in
$\Int M$. The intersection of any hexagonal tile with any
distinct tile is either empty or a common edge joining
distinct components of $\partial{M}$. Every cusp tile meets $\partial{M}$
in its two infinite sides while its compact side is
shared with one hexagonal tile. The nerve of the tiling
of $M$ is thus a tree $T$ with one trivalent vertex for each
hexagonal tile and one univalent vertex for each cusp
tile.\ss

The procedure is initialized by construction of a hexagonal tile that
meets the first three sides in the given
enumeration of sides. After the first three sides, for each successive new side, one more hexagonal tile $H$ is inductively constructed; $H$ meets the
new side and those two of the earlier sides that
are in a topological sense adjacent. This induction completes the construction of all the hexagonal tiles. To terminate the tiling procedure, the cusp
tiles are then defined to be the closures of the
components of the complement of the union of all the hexagonal
tiles. The cusp tiles correspond bijectively to the isolated ends of $M$.\ss

This \textsc{diff} tiling is well-defined by the given enumeration of the sides
of $M$, up to a \textsc{diff} isomorphism of tilings
that is piecewise \textsc{diff} isotopic to the identity of $M$.\ss

Each tile has a hyperbolic structure with the
length of each compact edge equal to one, and a
right angle at each vertex (infinity excepted).
After an isotopy of such structures,
they fit together to form a complete
hyperbolic structure $\sigma$ on $M$ making $\partial{M}$ geodesic.\ss

This hyperbolic structure $\sigma$ on $M$ is well-defined by the tiling, up to
isometry ambient isotopic to the identity.\ss

To conclude, one develops $M_\sigma$ isometricly into $\mathbb{H}^2$, proceeding inductively tile
by tile, climbing up the above tree $T$, to realize
$M$ as a linear gasket in $\mathbb{H}^2$. \qed
\end{sgp}

\begin{remark}
The hyperbolic structure $\sigma$ on $M$ obtained by the above tiling procedure is often distinct from any structure obtained by the classical proof; indeed for every isolated end of $M$ the limit points of its cusp tile neighborhood in the ideal circle at infinity $\partial{\mathbb{B}^2}$ of $\mathbb{H}^2$ constitute a whole compact interval rather than a point.
However, this clear geometric distinction can be suppressed as follows: the cut-locus in $M_\sigma$ of $\partial{M}$ is a properly embedded piecewise geodesic graph $\Gamma \subset \Int\, M$, which meets each tile in a standard way. The convex hull of the closure of $\Gamma$ in $\mathbb{B}^2$, intersected with $\mathbb{H}^2$, is a smaller but visibly diffeomorphic copy $M'$ of $M$ whose hyperbolic structure is of the sort obtained in the classical proof.
\end{remark}

\section{acknowledgments}
The authors thank R.D.~Edwards, J.-P.~Otal, and an anonymous referee for helpful comments.

\end{document}